\documentclass[11pt, oneside]{amsart} 

\usepackage[british,english]{babel} 
\usepackage{graphics,color,pgf}
\usepackage{epsfig}
\usepackage[ansinew]{inputenc}
\usepackage[all]{xy}
\usepackage{hyperref}
\newdir{ >}{!/8pt/@{}*@{>}}
\usepackage{amssymb, amsmath,amsthm, mathtools, amscd}
\usepackage{mathrsfs}
\usepackage{stmaryrd}
\usepackage[margin=1.5in]{geometry}

\makeatletter
\newtheorem*{rep@theorem}{\rep@title}
\newcommand{\newreptheorem}[2]{%
\newenvironment{rep#1}[1]{%
 \def\rep@title{#2 \ref{##1}}%
 \begin{rep@theorem}}%
 {\end{rep@theorem}}}
\makeatother

\newtheorem{intro_thm}{Theorem}

\newtheorem{intro_prop}[intro_thm]{Proposition}

\theoremstyle{plain}
\newtheorem{teor}{Theorem}[section]
\newreptheorem{teor}{Theorem}  
\newtheorem{lem}[teor]{Lemma}

\newtheorem{prop}[teor]{Proposition}
\newreptheorem{prop}{Proposition}  

\theoremstyle{definition}
\newtheorem{deft}[teor]{Definition}

\theoremstyle{remark}
\newtheorem{oss}[teor]{Remark}
\newtheorem{ntn}[teor]{Notation}

\DeclareMathOperator\upC{\textup{C}}

\DeclareMathOperator\upL{\textup{L}}

\DeclareMathOperator\bbH{\mathbb{H}}

\DeclareMathOperator\bbN{\mathbb{N}}

\DeclareMathOperator\bbR{\mathbb{R}}
\DeclareMathOperator\bbS{\mathbb{S}}

\DeclareMathOperator\bbZ{\mathbb{Z}}

\DeclareMathOperator\homeos{\textup{Homeo}^+(\mathbb{S}^1)}
\DeclareMathOperator\homeor{\textup{Homeo}^+_{\mathbb{Z}}(\mathbb{R})}
\DeclareMathOperator\overC{\overline{\textup{C}}}

\begin{document}

\title[Parametrized Euler class and semicohomology theory]{Parametrized Euler class and semicohomology}

\author[A. Savini]{A. Savini}
\address{Department of Mathematics, University of Milano-Bicocca, Via Roberto Cozzi 55, Milan 20125, Italy}
\email{alessio.savini@unimib.it}

\date{\noindent \today, \textbf{MSC Classification 2020}: 57M60, 22E40, \textbf{keywords}: bounded cohomology, Euler class, measurable cocycle, semicohomology}

\begin{abstract}

Given a measurable cocycle with values in $\textup{Homeo}^+(\mathbb{S}^1)$, 
we construct an $\textup{L}^\infty$-parametrized Euler class in bounded cohomology. We show that such a class vanishes if and only if the cocycle can be lifted to $\textup{Homeo}^+_{\mathbb{Z}}(\mathbb{R})$ and it admits an equivariant family of points. 

We recall the notion of semicohomologous cocycles and we show that two measurable cocycles are semicohomologous if and only if they induce the same parametrized Euler class. Since for minimal cocycles, semicohomology boils down to cohomology, the parametrized Euler classes of minimal cohomologous cocycles are equal. 

We conclude by studying the vanishing of the real parametrized Euler class and we obtain some results about elementary cocycles. 
\end{abstract}
  
\maketitle

\section{Introduction}

One of the most elementary and at the same time intriguing field in dynamics is the study of \emph{circle actions}. A circle action of a group $\Gamma$ is a representation $\rho:\Gamma \rightarrow \textup{Homeo}^+(\bbS^1)$, where $\homeos$ is the group of orientation preserving homeomorphisms of $\bbS^1$. In the particular case when $\Gamma=\bbZ$, the dynamics of circle actions is well-understood thanks to the notion of \emph{rotation number} studied by Poincar\'e \cite{Poi81,Poi82}. For more general groups one can say that either $\rho$ has a finite orbit, or each orbit is dense (the action is \emph{minimal}) or there exists an invariant Cantor set. 

As it usually happens in Mathematics, we should identify equivalent actions. Since we are dealing with representations, it could be natural to consider the usual conjugacy relation. However, in this particular context, it is more convenient to introduce a weaker relation, namely \emph{semiconjugacy}. Recall that two representations are semiconjugate if there exists a \emph{non-decreasing degree one map} $\varphi:\bbS^1 \rightarrow \bbS^1$ which intertwines the actions. 

A very powerful tool in the investigation of circle actions is given by bounded cohomology. The starting point of this approach is the \emph{bounded Euler class} $e^b_{\bbZ} \in \textup{H}^2_b(\homeos;\bbZ)$. Given a representation $\rho:\Gamma \rightarrow \homeos$, we can exploit the functoriality of bounded cohomology to pullback the class $e^b_{\bbZ}$, obtaining a class $\textup{H}^2_b(\rho)(e^b_{\bbZ}) \in \textup{H}^2_b(\Gamma;\bbZ)$. Ghys \cite{ghys:articolo} proved that $\textup{H}^2_b(\rho)(e^b_{\bbZ})$ is a complete invariant of the semiconjugacy class of the representation. Additionally, he also provided a criterion to understand which classes in $\textup{H}^2_b(\Gamma;\bbZ)$ can be realized as pullbacks of the bounded Euler class. 

More recently Burger \cite{Burger:ext} studied the extendability of the strongly proximal factor associated to a minimal unbounded action of a lattice $\Gamma \leq G$ (one can always modify a minimal action in its semiconjugacy class to obtain a strongly proximal one). Burger noticed that the extension property is deeply related to the pullback of the \emph{real} bounded Euler class $e^b_{\bbR}$ (the image of $e^b_\mathbb{Z}$ in $\textup{H}^2_b(\Gamma;\bbR)$).  More precisely, denoting by $\rho_{sp}$ the strongly proximal factor associated to $\rho$, we have that $\rho_{sp}$ can be extended to $G$ if and only if the pullback $\textup{H}^2_b(\rho)(e^b_{\bbR})$ lies in the restriction map induced by the inclusion of $\Gamma$ in $G$. Using this, together with results obtained with Monod \cite{BM1,burger2:articolo} about bounded cohomology groups of higher rank lattices, Burger proved several known rigidity results in a unified manner.  Similar rigidity results were previously obtained by Ghys \cite{Ghys1}, Witte and Zimmer \cite{WZ01}, Navas \cite{Nav05} and Bader, Furman and Shaker \cite{BFS06}. 

The purpose of this paper is to extend the application of bounded cohomology in the study of measurable cocycles with values in $\homeos$. The author, together with Moraschini and Sarti \cite{savini:surface,moraschini:savini,moraschini:savini:2,sarti:savini,sarti:savini:2,savini3:articolo}, has recently developed a machinery to define the pullback of a bounded cohomology class along a measurable cocycle. The invariant that we would obtain following the methods developed in those papers would be too rough to study the dynamics of a measurable cocycle. For this reason we are going to introduce the notion of \emph{parametrized Euler class}. Let $\Gamma$ be a countable group and let $(\Omega,\mu)$ be a standard Borel probability $\Gamma$-space. Starting from measurable cocycle $\sigma:\Gamma \times \Omega \rightarrow \homeos$, we can define a cohomology class with coefficient in $\textup{L}^\infty(\Omega;\bbZ)$ that we are going to call \emph{$\Omega$-parametrized Euler class}. Such a class, denoted by $\textup{H}^2_b(\sigma)(e^b_{\bbZ})$, will encode all the dynamical information associated to $\sigma$. For instance, it will vanish for those cocycles admitting an \emph{equivariant family of points} (Definition \ref{def:equivariant:family}). 

Examples of parametrized classes may arise naturally in the study of measure equivalence of surface groups. Indeed, one can associate to a measure equivalence cocycle its parametrized class and, in case of self-couplings, the integral along the parameter space corresponds to the Euler invariant defined by Bader, Furman and Sauer \cite{sauer:articolo}. 

Before stating our main result we need to introduce the correct equivalence relation on the space of cocycles. Bader, Furman and Shaker \cite{BFS06} suggested the notion of \emph{semicohomologous cocycles}, but their intertwining maps are continuous. Here we want to follow the line of Bucher, Frigerio and Hartnick \cite{BFH} and we introduce the notion of \emph{left semicohomology}. Given two measurable cocycles $\sigma_1, \sigma_2:\Gamma \times \Omega \rightarrow \homeos$, we say that $\sigma_2$ is \emph{left semicohomologous} to $\sigma_1$ if there exists a measurable family $\varphi(\omega):\bbS^1 \rightarrow \bbS^1$ of non-decreasing degree one maps, with $\omega \in \Omega$, such that $\varphi(\gamma \omega)\sigma_1(\gamma,\omega)=\sigma_2(\gamma,\omega)\varphi(\omega)$, for every $\gamma \in \Gamma$ and almost every $\omega \in \Omega$. Equivalently we say that $\sigma_1$ is \emph{right semicohomologous} to $\sigma_2$. Clearly $\sigma_1$ and $\sigma_2$ will be semicohomologous if they are both left and right semicohomologous. It is worth noticing that we are not aware of explicit examples distinguishing our definition from the one introduced by Bader, Furman, and Shaker. At the moment such a difference remains purely theoretical.

\begin{intro_thm}\label{teor:main:intro}
Let $\Gamma$ be a countable group and let $(\Omega,\mu)$ be a standard Borel probability $\Gamma$-space. Consider two measurable cocycles $\sigma_1,\sigma_2:\Gamma \times \Omega \rightarrow \textup{Homeo}^+(\bbS^1)$. Then the following hold:
\begin{enumerate}
	\item Assume that $\sigma_1$ is semicohomologous to $\sigma_2$. If $\sigma_1$ admits an equivariant family of points, the same holds for $\sigma_2$ and $\textup{H}^2_b(\sigma_1)(e^b_{\bbZ})=\textup{H}^2_b(\sigma_2)(e^b_{\bbZ})=0$. 
	\item If $\textup{H}^2_b(\sigma_1)(e^b_{\bbZ})=\textup{H}^2_b(\sigma_2)(e^b_{\bbZ})$, then $\sigma_1$ and $\sigma_2$ are semicohomologous.
 	\item Assume that $\Omega$ is $\Gamma$-ergodic. Suppose that $\sigma_1$ does not admit any equivariant family of points and that it is left semicohomologous to $\sigma_2$. Then $\textup{H}^2_b(\sigma_1)(e^b_{\bbZ})=\textup{H}^2_b(\sigma_2)(e^b_{\bbZ}) \neq 0$. 
\end{enumerate}
\end{intro_thm}

The above result extends \cite[Theorem A]{ghys:articolo} to the context of measurable cocycles.

Following the work by Furstenberg \cite{furst:articolo}, one can give a notion of \emph{minimality} also for measurable cocycles (Definition \ref{def:minimal:cocycle}). Since for minimal cocycles semicohomology boils down to cohomology, we have that the parametrized class is the same for minimal cohomologous cocycles. 

Another crucial step in our investigation is the study of the \emph{real} parametrized Euler class. The latter will be simply the parametrized Euler class seen with real coefficients, that is $\textup{H}^2_b(\sigma)(e^b_{\bbR}) \in \textup{H}^2_b(\Gamma;\textup{L}^\infty(\Omega,\bbR))$. We are going to show the following 

\begin{intro_prop}\label{prop:vanishing:real}
Let $\Gamma$ be a countable group and let $(\Omega,\mu)$ be a standard Borel probability $\Gamma$-space. Consider a measurable cocycle $\sigma:\Gamma \times \Omega \rightarrow \homeos$. It holds that $\textup{H}^2_b(\sigma)(e^b_{\bbR})=0$ if and only if $\sigma$ is semicohomologous to a cocycle taking values in the rotations subgroup $\textup{Rot}$. 
\end{intro_prop}

An immediate consequence is given by

\begin{intro_prop}\label{prop:vanishing:coefficients}
Let $\Gamma$ be a countable group and let $(\Omega,\mu)$ be a standard Borel probability $\Gamma$-space. If $\mathrm{H}^2_b(\Gamma,\mathrm{L}^\infty(\Omega,\bbR))=0$, every measurable cocycle $\sigma: \Gamma \times \Omega \rightarrow \homeos$ is semicohomologous to a cocycle taking values in the rotations subgroup $\mathrm{Rot}$. 
\end{intro_prop}

The hypothesis in the previous result is satisfied for instance when the second degree bounded cohomology of $\Gamma$ vanishes for all \emph{semi-separable dual coefficients}. Recent examples of those groups are given by the wreath products studied by Monod in \cite{monod:thompson}. For either the latter family of groups or amenable groups, the reason why the integrated Euler class introduced by Moraschini, Sarti and the author fails to distinguish semicohomology classes becomes even clearer. Indeed, since the integration over $\Omega$ commutes with the change of coefficients $\mathrm{L}^\infty(\Omega,\mathbb{Z}) \rightarrow \mathrm{L}^\infty(\Omega,\mathbb{R})$, the integrated Euler class of any measurable cocycle $\sigma:\Gamma \times \Omega \rightarrow \homeos$ must vanish identically. 

Proposition \ref{prop:vanishing:coefficients} allows us to obtain two results about elementary cocycles of higher rank lattices. To correctly state those results, recall that a lattice $\Gamma \leq \prod_{i=1}^k G_i$ in a product of locally compact groups is \emph{irreducible} if its projection in each $G_i$ is dense. Similarly, a standard Borel probability $G$-space is \emph{irreducible} if the product $G_i':=\prod_{j \neq i} G_j$ acts ergodically on $\Omega$. Finally, a measurable cocycle $\sigma:\Gamma \times \Omega \rightarrow \homeos$ is \emph{elementary} if there exists a $\sigma$-equivariant weak-$^\ast$ measurable map $\nu:\Omega \rightarrow \mathcal{M}^1(\mathbb{S}^1)$ in the space of probability measures on $\mathbb{S}^1$ (Definition \ref{def:elementary:cocycle}).

The following theorem should be compared to \cite[Theorem 22]{burger2:articolo} and \cite[Theorem 1.6]{BFS06} 

\begin{intro_thm}\label{teor:lattice:product:elementary}
Let $G=\prod_{i=1}^k G_i$, where $G_i$ is a locally compact second countable group and $k \geq 2$. Let $\Gamma \leq G$ be an irreducible lattice and let $(\Omega,\mu)$ be an irreducible standard Borel probability $G$-space. If no one of the $G_i$'s is Hermitian, then every measurable cocycle $\sigma:\Gamma \times \Omega \rightarrow \homeos$ is semicohomologous to a cocycle with values in $\textup{Rot}$. 
\end{intro_thm}

We have also a similar result in the case of lattices in higher rank simple Lie groups (compare it to \cite[Theorem 1]{Ghys1}, \cite[Theorem 5.4]{WZ01} and \cite[Theorem B]{Nav06}).

\begin{intro_thm}\label{teor:higher:rank:elementary}
Let $\kappa$ be a local field and let $\mathbf{G}$ be a $\kappa$-connected simple algebraic group defined over $\kappa$ with $\textup{rank}_\kappa(\mathbf{G}) \geq 2$. 
Let $G=\mathbf{G}(\kappa)$ be the group of $\kappa$-points of $\mathbf{G}$ and consider a lattice $\Gamma \leq G$. Let $(\Omega,\mu)$ be an ergodic standard Borel probability $\Gamma$-space. Suppose that $\textup{H}^2(\Gamma;\bbR)=0$. Then every measurable cocycle $\sigma:\Gamma \times \Omega \rightarrow \homeos$ is semicohomologous to a cocycle with values in $\textup{Rot}$. 
\end{intro_thm}

Thanks to Proposition \ref{prop:vanishing:coefficients}, the proof of both Theorem \ref{teor:lattice:product:elementary} and Theorem \ref{teor:higher:rank:elementary} boils down to show the vanishing of the group $\mathrm{H}^2_b(\Gamma,\mathrm{L}^\infty(\Omega,\bbR))$.

\subsection*{Plan of the paper}

Section \ref{sec:preliminary} is devoted to preliminary definitions. We recall the notion of both cohomology and bounded cohomology of discrete groups, then we move to some results about abelian extensions and we conclude with the definition of the bounded Euler class. 

In Section \ref{sec:parametrized:euler} we introduce the notion of parametrized Euler class and we show that it actually extends the definition given for representations. In Section \ref{sec:vanishing:euler} we focus our attention on the vanishing of the parametrized Euler class. We show that if the parametrized Euler class vanishes, the measurable cocycle is semicohomologous to the trivial action and equivalently it admits an equivariant family of points. The first statement of Theorem \ref{teor:main:intro} is proved. We conclude the proof of Theorem \ref{teor:main:intro} in Section \ref{sec:ghys:theorem}. We split it into Theorem \ref{teor:ghys:pullback:semicohomology} and Proposition \ref{prop:ghys:nonvanishing}.

In Section \ref{sec:minimal:cocycles} we give the notion of minimal cocycle and we show that for minimal cocycles left semicohomology and cohomology are equivalent. 

We conclude by studying the vanishing of the real parametrized Euler class. In the last section we prove Proposition \ref{prop:vanishing:real}, Theorem \ref{teor:lattice:product:elementary} and Theorem \ref{teor:higher:rank:elementary}.

\subsection*{Acknowledgements}

The author thanks the anonymous referee for a detailed report and valuable suggestions that helped to improve the quality of the manuscript.

\section{Preliminary definitions}\label{sec:preliminary}

\subsection{Bounded cohomology of discrete groups}\label{sec:bounded:cohomology}

In this section we are going to recall briefly the definition of cohomology and bounded cohomology for discrete groups. In particular the latter can be expressed by the use of both the homogeneous complex and the inhomogeneous one. We are going to see how they are related. We refer the reader to Frigerio's book \cite{miolibro} for a more detailed exposition about this topic.
 
Let $\Gamma$ be a discrete group. Let $A$ be a normed $\Gamma$-module and suppose that $\Gamma$ acts on $A$ by preserving the norm. We can define the space of $A$-valued functions on $\Gamma$ as
$$
\textup{C}^\bullet(\Gamma;A):=\{ f: \Gamma^{\bullet+1} \rightarrow A \} \ ,
$$
endowed with natural $\Gamma$-action given by
\begin{equation}\label{eq:action:cochain}
(\gamma f)(\gamma_0,\cdots,\gamma_\bullet):=\gamma \cdot (f(\gamma^{-1} \gamma_0,\cdots,\gamma^{-1}\gamma_\bullet)) \ ,
\end{equation}
for every $f \in \textup{C}^\bullet(\Gamma;A)$ and every $\gamma,\gamma_0,\cdots,\gamma_\bullet \in \Gamma$. The notation $\gamma \cdot$ refers to the $\Gamma$-action on $A$. We are going to say that an $A$-valued function is $\Gamma$-\emph{invariant} if it holds
$$
\gamma f=f \ ,
$$
for every $\gamma \in \Gamma$. We denote by $\textup{C}^\bullet(\Gamma;A)^\Gamma$ the submodule of $\Gamma$-invariant functions. Defining the \emph{standard homogeneous coboundary operator} $\delta^\bullet$ as
$$
\delta^\bullet:\textup{C}^\bullet(\Gamma;A) \rightarrow \textup{C}^{\bullet+1}(\Gamma;A) \ ,
$$
$$
(\delta^\bullet f)(\gamma_0,\cdots,\gamma_{\bullet+1}):=\sum_{i=0}^{\bullet+1} (-1)^i f(\gamma_0,\cdots,\gamma_{i-1},\gamma_{i+1},\cdots,\gamma_{\bullet+1}) \ ,
$$
for every $\gamma_0,\cdots,\gamma_{\bullet+1} \in \Gamma$, we obtain a cochain complex $(\textup{C}^\bullet(\Gamma;A),\delta^\bullet)$. The coboundary operator respects $\Gamma$-invariance and hence it restricts to the collection of submodules of $\Gamma$-invariant functions. 

\begin{deft}\label{def:ordinary:cohomology}
Let $\Gamma$ be a discrete group and let $A$ be a normed $\Gamma$-module. The \emph{cohomology} of the group $\Gamma$ with coefficients in $A$ is the cohomology of the complex $(\textup{C}^\bullet(\Gamma;A)^\Gamma,\delta^\bullet)$ and it is denoted by $\textup{H}^\bullet(\Gamma;A)$. 
\end{deft}

The additional datum of a norm on $A$ allows us to define the notion of $A$-valued bounded functions. Indeed, given a function $f \in \textup{C}^\bullet(\Gamma;A)$, there exists a natural norm given by 
$$
\lVert f \rVert_\infty := \sup_{\gamma_0,\cdots,\gamma_\bullet \in \Gamma} |f(\gamma_0,\cdots,\gamma_\bullet)| \ ,
$$
and we are going to say that $f$ is \emph{bounded} if its norm is finite. We denote by $\textup{C}^\bullet_b(\Gamma;A) \subset \textup{C}^\bullet(\Gamma;A)$ the submodule of bounded functions. Since the coboundary operator $\delta^\bullet$ is defined by a finite sum, it naturally preserves boundedness and it can be restricted to the submodules of bounded functions.

\begin{deft}\label{def:bounded:cohomology}
	Let $\Gamma$ be a discrete group and let $A$ be a normed $\Gamma$-module on which $\Gamma$ acts isometrically. The \emph{bounded cohomology} of the group $\Gamma$ with coefficients in $A$ is the cohomology of the complex $(\textup{C}^\bullet_b(\Gamma;A)^\Gamma,\delta^\bullet)$ and it is denoted by $\textup{H}^\bullet_b(\Gamma;A)$. 
\end{deft}

The natural inclusion of bounded functions, namely
$$
\iota^\bullet_\Gamma: \textup{C}^\bullet_b(\Gamma;A) \rightarrow \textup{C}^\bullet(\Gamma;A) \ ,
$$
is a cochain map and hence induces a map at the level of cohomology groups
$$
\textup{comp}^\bullet_\Gamma:\textup{H}^\bullet_b(\Gamma;A) \rightarrow \textup{H}^\bullet(\Gamma;A) \ ,
$$
called \emph{comparison map}. 

One may want to get rid of the $\Gamma$-invariance property. In order to do this, one needs to introduce the \emph{inhomogeneous complex} of $A$-valued (bounded) functions. More precisely, we define 
$$
\overC^\bullet_{(b)}(\Gamma;A):=\textup{C}^{\bullet-1}_{(b)}(\Gamma;A) \ ,
$$
which is isomorphic to the module of $\Gamma$-invariant $A$-valued (bounded) functions. Following \cite[Proposition 7.4.12]{monod:libro}, the two desired isomorphisms are given below
$$
V^\bullet: \overC^\bullet_{(b)}(\Gamma;A) \rightarrow \textup{C}^\bullet_{(b)}(\Gamma;A)^\Gamma \ , \ \ (V^\bullet f)(\gamma_0,\cdots,\gamma_\bullet):=\gamma_0 \cdot(f(\gamma_0^{-1}\gamma_1,\cdots,\gamma_{\bullet-1}^{-1}\gamma_\bullet)) \ ,
$$
$$
W^\bullet:\textup{C}^\bullet_{(b)}(\Gamma;A)^\Gamma \rightarrow \overC^\bullet_{(b)}(\Gamma;A) \ , \ \ (W^\bullet f)(\gamma_1,\cdots,\gamma_{\bullet}):=f(e_\Gamma,\gamma_1,\gamma_1 \gamma_2, \cdots, \gamma_1 \cdots \gamma_\bullet) \ ,
$$
where $e_\Gamma \in \Gamma$ is the neutral element.

We can rewrite the coboundary operator in terms of the inhomogeneous modules. More precisely, by suitably applying the maps $V^\bullet,W^\bullet$ defined above, we get
$$
\overline{\delta}^\bullet:\overC^\bullet_{(b)}(\Gamma;A) \rightarrow \overC^{\bullet+1}_{(b)}(\Gamma;A) \ ,
$$
\begin{small}
$$
(\overline{\delta}^\bullet f)(\gamma_0,\cdots,\gamma_{\bullet}):=\gamma_0 \cdot (f(\gamma_1,\cdots,\gamma_\bullet))+\sum_{i=1}^\bullet (-1)^if(\gamma_0,\cdots,\gamma_i \gamma_{i+1},\cdots,\gamma_{\bullet})+(-1)^nf(\gamma_0,\cdots,\gamma_{\bullet-1}) \ .
$$
\end{small}

In this way we obtain a cochain complex $(\overC^\bullet_{(b)}(\Gamma;A),\overline{\delta}^\bullet)$ which computes the (bounded) cohomology $\textup{H}^\bullet_{(b)}(\Gamma;A)$. In what follows we will mainly use the inhomogeneous complex for computations.

\subsection{Abelian extensions of a discrete group}\label{sec:extension}

In this section we will briefly recall the relation that exists between the abelian extensions of a discrete group $\Gamma$ via a $\Gamma$-module $A$ and the associated cohomology group $\textup{H}^2(\Gamma;A)$. We are going to see that any cocycle in $\overC^2(\Gamma;A)$ is associated to an extension and that if two cocycles differ by a coboundary then the associated extensions are isomorphic. We refer the reader to \cite[Chapter IV, Section 3]{brown} for a broad discussion about this correspondence. 

Let $\Gamma$ be a discrete group and let $A$ be an abelian group. We say that the group $E$ is an \emph{extension} of $\Gamma$ by $A$ if there exists a short exact sequence 
$$
\xymatrix{
0 \ar[r] & A \ar[r]^i & E \ar[r]^p & \Gamma \ar[r] & 0 \ ,
}
$$
where both $i$ and $p$ are morphisms of groups. Such an extension of $\Gamma$ by $A$ determines a $\Gamma$-module structure on $A$ as follows:  one can define $\gamma \cdot a= \widetilde{\gamma}^{-1} a \widetilde{\gamma}$, where $\gamma \in \Gamma, a \in A$ and $\widetilde{\gamma} \in E$ is such that $p(\widetilde{\gamma})=\gamma$. The action is well-defined because $A$ is both abelian and normal in $E$. We warn the reader that, in the case that $A$ has already a $\Gamma$-module structure, we will consider extensions that give rise to the same $\Gamma$-module structure. 

Suppose we have a set-theoretic section to $p$, that is a map $s:\Gamma \rightarrow E$, such that $p \circ s=\textup{id}_\Gamma$. We are going to say that $s$ is \emph{normalized} if $s(e_\Gamma)=e_E$, where $e_\Gamma$ and $e_E$ are the neutral elements of the corresponding groups. 

Given a normalized section we can construct an element $\varepsilon \in \overC^2(\Gamma;A)$ as follows. Take $\gamma,\lambda \in \Gamma$ and consider the elements $s(\gamma)s(\lambda)$ and $s(\gamma\lambda)$. Since both the elements are projected by $p$ to the product $\gamma\lambda$, there exists a unique element $\varepsilon(\gamma,\lambda) \in A$ such that 
$$
s(\gamma\lambda)i(\varepsilon(\gamma,\lambda))=s(\gamma)s(\lambda) \ .
$$
Notice that the normalization assumption on the section $s$ implies that 
\begin{equation}\label{eq:normalization}
\varepsilon(\gamma,e_\Gamma)=\varepsilon(e_\Gamma,\gamma)=1_A \ ,
\end{equation}
for every $\gamma \in \Gamma$. 

The function $\varepsilon$, together with the $\Gamma$-module structure on $A$, determines completely the algebraic structure of the extension $E$. Indeed from a set-theoretic point of view $E$ can be identified with the product $ \Gamma \times A$ endowed with the group law 
\begin{equation}\label{eq:group:law:cocycle}
(s(\gamma)i(a))\cdot(s(\lambda)i(b))=s(\gamma\lambda)i(\varepsilon(\gamma,\lambda)+(\lambda^{-1} \cdot a)+b) \ ,
\end{equation}
where $\gamma,\lambda \in \Gamma$, $a,b \in A$ and $\lambda^{-1} \cdot a$ denotes the action of $\Gamma$ on $A$. The associativity property for the group law given by Equation \eqref{eq:group:law:cocycle} implies that $\varepsilon$ is cocycle, thus it defines a class in the group $\textup{H}^2(\Gamma;A)$. 

Additionally one can check that if we choose another normalized section $s':\Gamma \rightarrow E$ we will obtain a cocycle $\varepsilon'$ which is actually cohomologous to the one induced by $s$. In this way we obtain a bijective correspondence between the isomorphism classes of extensions of $\Gamma$ by $A$ and the group $\textup{H}^2(\Gamma;A)$. Recall that two extensions $E,E'$ of $\Gamma$ by $A$ are \emph{isomorphic} if there exists the following commutative diagram
$$
\xymatrix{
 	  &                            &  E \ar[dr]^p \ar[dd]^{\cong} &                         &   \\
0 \ar[r] & A \ar[ru]^i \ar[rd]_{i'} &                              & \Gamma \ar[r] & 0. \\
            &                            &  E' \ar[ru] _{p'}            &                         &
}
$$

\subsection{Bounded Euler class}\label{sec:bounded:euler:class}

Fix an orientation on the circle $\bbS^1$. There exists a unique orientation on $\bbR$ such that the covering map $\pi:\bbR \rightarrow \bbS^1 \cong \bbR/\bbZ$ is orientation preserving. 

We denote the group of orientation preserving homeomorphisms of the circle by $H:=\textup{Homeo}^+(\bbS^1)$, endowed with the usual compact-open topology. We can consider its universal cover $\widetilde{H}:=\textup{Homeo}^+_{\bbZ}(\bbR)$, which is defined as
$$
\textup{Homeo}^+_{\bbZ}(\bbR) : = \{ f: \bbR \rightarrow \bbR \ | \ \textup{$f$ is orientation preserving and $f(x+1)=f(x)+1$} \} \ ,
$$
endowed with the compact-open topology, as well. The translation $\tau(x):=x+1$ is an element of $\widetilde{H}$ which is central by definition. More precisely, the group generated by $\tau$ is exactly the center of $\widetilde{H}$ and we have a short exact sequence of groups
$$
\xymatrix{
0 \ar[r] & \bbZ \ar[r]^i & \widetilde{H} \ar[r]^p & H \ar[r]  \ar@/_1pc/[l]_{\overline{s}} & 0 \ .
}
$$

Consider now a normalized section $\overline{s}:H \rightarrow \widetilde{H}$ defined by $\overline{s}(f)(0) \in [0,1)$. Notice that the action of $\widetilde{H}$ on $i(\bbZ)$ is trivial, since the latter is central, and this allows to see the integers as a trivial $H$-module. For every $f,g \in H$, there exists an integer $\overline{\varepsilon}(f,g)$ such that
\begin{equation}\label{eq:normalized:euler:cocycle}
\overline{s}(f,g)i(\overline{\varepsilon}(f,g))=\overline{s}(f)\overline{s}(g) \ ,
\end{equation}
and the function $\overline{\varepsilon}:H \times H \rightarrow \bbZ$ satisfies $\overline{\delta}^2 \overline{\varepsilon}=0$. Moreover, as noticed by Ghys \cite{ghys:articolo} the function $\overline{\varepsilon}$ is actually bounded and hence $\overline{\varepsilon} \in \overC^2_b(H;\bbZ)$. For a topological group $G$ and a normed $G$-module $A$, we set $\mathrm{H}^\bullet_b(G,A):=\mathrm{H}^\bullet_b(G^\delta,A)$, where $G^\delta$ refers to $G$ endowed with the discrete topology. 

\begin{deft}\label{def:euler:class:bounded:euler:class}
The cohomology class induced by $\overline{\varepsilon}$ is called \emph{Euler class} and we are going to denote it by $e_{\bbZ} \in \textup{H}^2(H;\bbZ)$. We are going to call its bounded analogue \emph{bounded Euler class} and we are going to denote it by $e^b_{\bbZ} \in \textup{H}^2_b(H;\bbZ)$. 
\end{deft}

In Section \ref{sec:real:euler:class} we are going to study also the notion of bounded Euler class with \emph{real} coefficients. Such a class is obtained by considering the function $\overline{\varepsilon}$ as a real-valued function rather than an integer-valued one. We are going to denote such a class by $e^b_{\bbR} \in \textup{H}^2_b(H;\bbR)$.

\section{The parametrized Euler class}\label{sec:parametrized:euler}

In this section we are going to introduce the main character of the paper, that is the parametrized Euler class. This will be a natural generalization of the standard Euler class with integral coefficients coming from representation theory. 

Let $\Gamma$ be a countable group and let $(\Omega,\mu)$ be a standard Borel probability space. We assume that the $\Gamma$-action on $\Omega$ is essentially free (that is almost every stabilizer is trivial) and measure preserving. We are going to call $(\Omega,\mu)$ a \emph{standard Borel probability $\Gamma$-space}. If $(\Theta,\nu)$ is another measure space, we are going to denote by $\upL^0(\Omega,\Theta)$ the space of \emph{equivalence classes of measurable functions}, where two functions are identified if they coincide up to a measure zero set. We endow $\upL^0(\Omega,\Theta)$ with the topology of the \emph{convergence in measure}. 

\begin{ntn}\label{ntn:L0}
From now until the end of the paper, we are going to pick an everywhere defined representative of a class in $\upL^0(\Omega,\Theta)$. Additionally, if we have an equation in $\upL^0(\Omega,\Theta)$, the latter corresponds to an equation at the level of representatives which holds almost everywhere. 

To build a representative, every time that we have an almost everywhere defined measurable fuction, we will tacitly consider a measurable extension everywhere. 
\end{ntn}

\begin{deft}\label{def:measurable:cocycle}
Let $\Gamma$ be a countable group and let $(\Omega,\mu)$ a standard Borel probability $\Gamma$-space. Given a topological group $H$ endowed with its Borel measurable structure, we say that a measurable map $\sigma:\Gamma \times \Omega \rightarrow H$ is a \emph{measurable cocycle} if it satisfies 
\begin{equation}\label{eq:measurable:cocycle}
\sigma(\gamma_1 \gamma_2, \omega)=\sigma(\gamma_1,\gamma_2\omega)\sigma(\gamma_2,\omega) \ ,
\end{equation}
for every $\gamma_1,\gamma_2 \in \Gamma$ and almost every $\omega \in \Omega$. 
\end{deft}

Following Notation \ref{ntn:L0}, the choice of a measurable cocycle amounts to selecting a representative of an equivalence class in $\mathrm{L}^0(\Gamma \times \Omega,H)$, where $\Gamma \times \Omega$ is equipped with the natural product structure. 

Equation \eqref{eq:measurable:cocycle} corresponds to the cocycle condition in the language of groupoids: in this context we are considering the semidirect groupoid associated to the $\Gamma$-action on $\Omega$. 

Recall that $H=\homeos$ and $\widetilde{H}=\homeor$. From Section \ref{sec:bounded:euler:class}, we know that there exists a short exact sequence
\begin{equation}\label{eq:short:sequence}
\xymatrix{
0 \ar[r] & \bbZ \ar[r]_i & \widetilde{H} \ar[r]_p & H \ar[r] \ar@/_1pc/[l]_{\overline{s}} & 0 \ , 
}
\end{equation}
where the normalized section $\overline{s}$ is defined by setting  $\overline{s}(f)(0) \in [0,1)$. Notice that all the maps involved in the previous sequence are actually measurable with respect to the measurable Borel structure on each group. By applying the functor $\textup{L}^0(\Omega, \ \cdot \ )$ to the sequence given in Equation \eqref{eq:short:sequence}, we obtain another short exact sequence of groups
\begin{equation}\label{equation:short:sequence}
\xymatrix{
0 \ar[r] & \textup{L}^0(\Omega,\bbZ) \ar[r]_{i_\Omega} & \textup{L}^0(\Omega,\widetilde{H}) \ar[r]_{p_\Omega} & \textup{L}^0(\Omega,H) \ar[r] \ar@/_1pc/[l]_{\overline{s}_\Omega} & 0 \ , 
}
\end{equation}
where all the maps with the subscript $\Omega$ are obtained by composition with the corresponding maps in Equation \eqref{eq:short:sequence}. For instance $i_\Omega(f)(\omega)=i(f(\omega))$ and the same for the other maps. Here the abelian group $\mathrm{L}^0(\Omega,\mathbb{Z})$ and the non-abelian groups $\mathrm{L}^0(\Omega,\widetilde{H})$, $\mathrm{L}^0(\Omega,H)$ are endowed with the natural structures induced pointwise by $\mathbb{Z}$, $\widetilde{H}$, and $H$, respectively.

In Section \ref{sec:bounded:euler:class} we saw that there exists a function 
$$
\overline{\varepsilon}: H^2 \rightarrow \bbZ \ , \ \ \ \overline{\varepsilon}(h_1,h_2):=\overline{s}(h_1h_2)^{-1}\overline{s}(h_1)\overline{s}(h_2) \ ,
$$
which satisfies $\overline{\delta}^2 \overline{\varepsilon}=0$. This leads to the definition of the bounded Euler class $e^b_{\bbZ} \in \textup{H}^2_b(H;\bbZ)$. Let $\sigma:\Gamma \times \Omega \rightarrow H$ be a measurable cocycle. We want to show how we can produce a parametrized Euler class. The latter will be nothing else than the (non integrated) pullback of $e^b_{\bbZ}$ along the measurable cocycle $\sigma$. More precisely we can define 
$$
\overline{\textup{C}}^2_b(\sigma)(\overline{\varepsilon}): \Gamma^2 \rightarrow \textup{L}^\infty(\Omega,\bbZ) \ , 
$$
\begin{equation}\label{eq:euler:pullback}
(\overline{\textup{C}}_b^2(\sigma)(\overline{\varepsilon})(\gamma,\lambda))(\omega):=\overline{\varepsilon}(\sigma(\lambda^{-1},\omega)^{-1},\sigma(\lambda^{-1},\gamma^{-1}\omega)^{-1}) \ ,
\end{equation}
for every $\gamma,\lambda \in \Gamma$ and for almost every $\omega \in \Omega$. 

\begin{lem}\label{lem:cochain:map}
Let $\Gamma$ be a countable group and let $(\Omega,\mu)$ be a standard Borel probability $\Gamma$-space. Given a measurable cocycle $\sigma:\Gamma \times \Omega \rightarrow H$, the map 
$$
\overline{\textup{C}}^2_b(\sigma):\overline{\textup{C}}^2_b(H;\bbZ) \rightarrow \overline{\textup{C}}^2_b(\Gamma;\textup{L}^\infty(\Omega,\mathbb{Z})) \ , 
$$
\begin{equation}\label{eq:pullback:cochain}
\psi \mapsto \overline{\textup{C}}^2_b(\sigma)(\psi)(\gamma,\lambda)(\omega):=\psi(\sigma(\gamma^{-1},\omega)^{-1},\sigma(\lambda^{-1},\gamma^{-1}\omega)^{-1}) \ ,
\end{equation}
induces a well-defined map at the level of bounded cohomology groups
$$
\textup{H}^2_b(\sigma):\textup{H}^2_b(H;\bbZ) \rightarrow \textup{H}^2_b(\Gamma;\textup{L}^\infty(\Omega,\bbZ)) \ , \ \ \textup{H}^2_b(\sigma)([\psi]):=[\overline{\textup{C}}^2_b(\sigma)(\psi)] \ . 
$$
\end{lem}

\begin{proof}
Consider a map $\psi \in \overline{\textup{C}}^2_b(H;\bbZ)$. Suppose that it is a cocycle, that means 
\begin{equation}\label{eq:alexander:spanier}
\overline{\delta}^2\psi(h_1,h_2,h_3)=\psi(h_2,h_3)-\psi(h_1h_2,h_3)+\psi(h_1,h_2h_3)-\psi(h_1,h_2)=0 \ ,
\end{equation}
for every $h_1,h_2,h_3 \in H$. We need to show that $\overline{\textup{C}}^2_b(\sigma)(\psi)$ is a cocycle. Consider $\gamma_1,\gamma_2,\gamma_3 \in \Gamma$. For almost every $\omega \in \Omega$ we have that
\begin{align*}
\overline{\delta}^2 \overline{\textup{C}}^2_b(\sigma)(\psi)(\gamma_1,\gamma_2,\gamma_3)(\omega)=&\overC^2_b(\sigma)(\psi)(\gamma_2,\gamma_3)(\gamma_1^{-1}\omega)-\overC_b^2(\sigma)(\psi)(\gamma_1\gamma_2,\gamma_3)(\omega)\\
+&\overC^2_b(\sigma)(\psi)(\gamma_1,\gamma_2\gamma_3)(\omega)-\overC^2_b(\sigma)(\psi)(\gamma_1,\gamma_2)(\omega)=(\bullet) \ ,
\end{align*}
where we used jointly the definition of the coboundary map $\overline{\delta}^2$ and the action of $\Gamma$ on $\textup{L}^\infty(\Omega,\bbZ)$.  If we now plug in the definition of the map $\overC^2_b(\sigma)$ we obtain 
\begin{align*}
(\bullet)=&\psi(\sigma(\gamma_2^{-1},\gamma_1^{-1}\omega)^{-1},\sigma(\gamma_3^{-1},\gamma_2^{-1}\gamma_1^{-1}\omega)^{-1})-\psi(\sigma(\gamma_2^{-1}\gamma_1^{-1},\omega)^{-1},\sigma(\gamma_3^{-1},\gamma_2^{-1}\gamma_1^{-1}\omega)^{-1})\\
+&\psi(\sigma(\gamma_1^{-1},\omega)^{-1},\sigma(\gamma_3^{-1}\gamma_2^{-1},\gamma_1^{-1}\omega)^{-1})-\psi(\sigma(\gamma_1^{-1},\omega)^{-1},\sigma(\gamma_2^{-1},\gamma_1^{-1}\omega)^{-1}) \ .
\end{align*}
Applying Equation \eqref{eq:measurable:cocycle} to the second term and the third one appearing in the above sum and exploiting the fact that $\psi$ satisfies Equation \eqref{eq:alexander:spanier}, we obtain that $\overline{\delta}^2 \overline{\textup{C}}^2_b(\sigma)(\psi)(\gamma_1,\gamma_2,\gamma_3)$ is a cocycle in $\overline{\upC}^2_b(\Gamma;\upL^\infty(\Omega;\bbZ))$. 

If we define
$$
\overline{\textup{C}}^1_b(\sigma):\overline{\textup{C}}^1_b(H;\bbZ) \rightarrow \overline{\textup{C}}^1_b(\Gamma;\textup{L}^\infty(\Omega,\bbZ)) \ ,$$
$$
u \mapsto \overline{\upC}^1_b(\sigma)(u)(\gamma)(\omega):=u(\sigma(\gamma^{-1},\omega)^{-1}) \ ,
$$
we need to check that 
$$
\overline{\textup{C}}^2_b(\sigma)(\overline{\delta}^1 u)(\gamma,\lambda)(\omega)=(\overline{\delta}^1\overline{\upC}^1_b(\sigma)(u))(\gamma,\lambda)(\omega) \ . 
$$ 
By definition we have that 
\begin{align*}
\overline{\upC}^2_b(\sigma)(\overline{\delta}^1u)(\gamma,\lambda)(\omega)=&(\overline{\delta}^1u)(\sigma(\gamma^{-1},\omega)^{-1},\sigma(\lambda^{-1},\gamma^{-1}\omega)^{-1})=\\
=&u(\sigma(\lambda^{-1},\gamma^{-1}\omega)^{-1})-u(\sigma(\gamma^{-1},\omega)^{-1}\sigma(\lambda^{-1},\gamma^{-1}\omega)^{-1})+u(\sigma(\gamma^{-1},\omega)^{-1})=\\
=&(\overline{\delta}^1\overline{\upC}^1_b(\sigma)(u))(\gamma,\lambda)(\omega) \ ,
\end{align*}
where we exploited Equation \eqref{eq:measurable:cocycle} to move from the second line to the third one. This concludes the proof. 
\end{proof}

Lemma \ref{lem:cochain:map} enables us to define the notion of parametrized Euler class.

\begin{deft}\label{def:parametrized:euler:class}
Let $\Gamma$ be a countable group and let $(\Omega,\mu)$ be a standard Borel probability $\Gamma$-space. Given a measurable cocycle $\sigma:\Gamma \times \Omega \rightarrow H$, the \emph{$\Omega$-parametrized Euler class} associated to $\sigma$ is the class $\textup{H}^2_b(\sigma)(e^b_{\bbZ}) \in \textup{H}^2_b(\Gamma;\textup{L}^\infty(\Omega,\bbZ))$. 
\end{deft}

The name of \emph{$\Omega$-parametrized class} suggests that the pullback of $\overline{\varepsilon}$ along the measurable cocycle $\sigma$ is actually a measurable function whose parameter space is precisely $\Omega$.

\begin{oss}\label{oss:not:bounded}
The same definition given in Equation \eqref{eq:pullback:cochain} still holds if the cochains are not bounded. This leads to a well-defined map 
$$
\textup{H}^2(\sigma):\textup{H}^2(H;\bbZ) \rightarrow \textup{H}^2(\Gamma;\textup{L}^\infty(\Omega,\bbZ)) \ ,
$$
which fits into the following commutative diagram
$$
\xymatrix{
\textup{H}^2_b(H;\bbZ) \ar[d]^{\textup{comp}^2_{H}} \ar[rrr]^{\textup{H}^2_b(\sigma)} &&& \textup{H}^2_b(\Gamma;\textup{L}^\infty(\Omega,\bbZ)) \ar[d]^{\textup{comp}^2_\Gamma} \\ 
\textup{H}^2(H;\bbZ) \ar[rrr]^{\textup{H}^2(\sigma)} &&& \textup{H}^2(\Gamma;\textup{L}^\infty(\Omega,\bbZ)) \ ,
}
$$
where $\textup{comp}^2$ is the comparison map introduced in Section \ref{sec:bounded:cohomology}. 

A priori, to distinguish the classes $\textup{H}^2(\sigma)(e_{\bbZ})$ and $\textup{H}^2_b(\sigma)(e^b_{\bbZ})$, we could have called the latter parametrized \emph{bounded} Euler class. Since we do not want to make the notation heavier, we will actually refer to both using the same name with an abuse of notation.
\end{oss}

Now we want to show that the definition we gave of parametrized Euler class is a generalization of the pullback along a representation. We start noticing that given a representation $\rho:\Gamma \rightarrow H$ we can construct naturally a cocycle as follows
$$
\sigma_\rho:\Gamma \times \Omega \rightarrow H \ , \ \ \ \sigma_\rho(\gamma,\omega):=\rho(\gamma) \ ,
$$
for every $\gamma \in \Gamma$ and every $\omega \in \Omega$. Here $(\Omega,\mu)$ is any standard Borel probability $\Gamma$-space. The fact that $\sigma_\rho$ is a cocycle is a direct consequence of the homomorphism rule for $\rho$.

\begin{lem}\label{lem:representation:cocycle}
Let $\Gamma$ be a countable group and let $(\Omega,\mu)$ be a standard Borel probability $\Gamma$-space. Given a representation $\rho:\Gamma \rightarrow H$, it holds
$$
\textup{H}^2_b(\sigma_\rho)(e^b_{\bbZ})=\textup{H}^2_b(\rho)(e^b_{\bbZ}) \ ,
$$
where $\sigma_\rho:\Gamma \times \Omega \rightarrow H$ is the measurable cocycle associated to $\rho$ and $\textup{H}^2_b(\rho)$ is the pullback defined by $\rho$. 
\end{lem}

Notice that, with an abuse of notation, we did not write the change of coefficients $\bbZ \rightarrow \mathrm{L}^\infty(\Omega,\bbZ)$ on the right-hand side. 

\begin{proof}
We are going to prove that $\overC^2_b(\sigma_\rho)=\overC^2_b(\rho)$. From this the statement will follow. 

Let $\gamma,\lambda \in \Gamma$ and let $\omega \in \Omega$. Given any cochain $\psi \in \overC^2_b(H;\bbZ)$, we have that
\begin{align*}
\overC^2_b(\sigma_\rho)(\psi)(\gamma,\lambda)(\omega)&=\psi(\sigma_\rho(\gamma^{-1},\omega)^{-1},\sigma_\rho(\lambda^{-1},\gamma^{-1}\omega)^{-1})=\\
&=\psi(\rho(\gamma),\rho(\lambda))=\overC^2_b(\rho)(\psi)(\gamma,\lambda) \ .
\end{align*}
The above computation shows that $\overC^2_b(\sigma_\rho)(\psi)$ does not depend on $\omega \in \Omega$ and it coincides with $\overC^2(\rho)(\psi)$. This finishes the proof. 
\end{proof}

\section{Vanishing of the parametrized Euler class}\label{sec:vanishing:euler}

In this section we are going to study the vanishing of the parametrized Euler class. We are going to see that if the parametrized Euler class vanishes, the cocycle is liftable and there exists an equivariant family of points on the circle. We refer the reader to \cite[Section 4.1]{BFH} for the same results in the case of representations. 

\begin{lem}\label{lem:cocycle:lift}
Let $\Gamma$ be a countable group and let $(\Omega,\mu)$ be a standard Borel probability $\Gamma$-space. Given a measurable cocycle $\sigma:\Gamma \times \Omega \rightarrow H$, then $\textup{H}^2(\sigma)(e_{\bbZ})$ vanishes if and only if there exists a measurable cocycle $\widetilde{\sigma}:\Gamma \times \Omega \rightarrow \widetilde{H}$ such that the following diagram commutes
$$
\xymatrix{
0 \ar[r] & \textup{L}^0(\Omega,\bbZ) \ar[r]_{i_\Omega} & \textup{L}^0(\Omega,\widetilde{H}) \ar[r]_{p_\Omega} & \textup{L}^0(\Omega,H) \ar[r] & 0 \\
& & & \Gamma \ar[u]^{\sigma} \ar@{.>}[ul]^{\widetilde{\sigma}} & .
}
$$
\end{lem}

\begin{proof}
Recall that $\textup{H}^2(\sigma)$ is the map of Remark \ref{oss:not:bounded}. Suppose that $\textup{H}^2(\sigma)(e_{\bbZ})$ vanishes. We are going to write an explicit candidate for the lift of $\sigma$. 

Since $\textup{H}^2(\sigma)(e_{\bbZ})=0$ there must exist a function 
$$
u:\Gamma \rightarrow \textup{L}^\infty(\Omega,\bbZ) \ ,
$$
such that 
$$
\overC^2(\sigma)(\overline{\varepsilon})(\gamma,\lambda)=(\overline{\delta}^1u)(\gamma,\lambda) \ ,
$$
where the equality holds as functions of $\textup{L}^\infty(\Omega,\bbZ)$. This means that for almost every $\omega \in \Omega$ it must hold
\begin{equation}\label{eq:euler:pullback:coboundary}
\overC^2(\sigma)(\overline{\varepsilon})(\gamma,\lambda)(\omega)=(\overline{\delta}^1u)(\gamma,\lambda)(\omega)=u(\lambda)(\gamma^{-1}\omega)-u(\gamma\lambda)(\omega)+u(\gamma)(\omega) \ .
\end{equation}

Exploiting Equation \eqref{eq:normalized:euler:cocycle} we can write explicitly the definition of $\overline{\varepsilon}$ obtaining 
\begin{equation}\label{eq:euler:pullback:section}
\overC^2(\sigma)(\overline{\varepsilon})(\gamma,\lambda)(\omega)=\overline{s}(\sigma(\lambda^{-1}\gamma^{-1},\omega)^{-1})^{-1}\overline{s}(\sigma(\gamma^{-1},\omega)^{-1})\overline{s}(\sigma(\lambda^{-1},\gamma^{-1}\omega)^{-1}) \ ,
\end{equation}
where $\overline{s}:H \rightarrow \widetilde{H}$ is the usual section defined by $\overline{s}(f)(0) \in [0,1)$. Comparing Equations \eqref{eq:euler:pullback:coboundary} and \eqref{eq:euler:pullback:section} we obtain
\begin{align}\label{eq:euler:comparison}
&\overline{s}(\sigma(\lambda^{-1}\gamma^{-1},\omega)^{-1})^{-1}\overline{s}(\sigma(\gamma^{-1},\omega)^{-1})\overline{s}(\sigma(\lambda^{-1},\gamma^{-1}\omega)^{-1})=\\ =&i(u(\lambda)(\gamma^{-1}\omega))i(-u(\gamma \lambda)(\omega))i(u(\gamma)(\omega)) \ , \nonumber
\end{align}
where $i:\bbZ \rightarrow \widetilde{H}$ is the inclusion. Then we can define 
\begin{equation}\label{eq:lift:formula}
\widetilde{\sigma}:\Gamma \times \Omega \rightarrow H \ , \ \ \ \widetilde{\sigma}(\gamma^{-1},\omega)^{-1}:=\overline{s}(\sigma(\gamma^{-1},\omega)^{-1})i(-u(\gamma)(\omega)) \ .
\end{equation}
Being the composition of measurable functions, $\widetilde{\sigma}$ is clearly measurable. We need only to show that it satisfies the cocycle equation. Recall that $i(\bbZ)$ is central in $\widetilde{H}$ and hence it commutes with the elements of $\overline{s}(H)$. Hence we can exploit Equation \eqref{eq:euler:comparison} to show that 
\begin{align*}
\widetilde{\sigma}(\lambda^{-1}\gamma^{-1},\omega)^{-1}&=\overline{s}(\sigma(\lambda^{-1}\gamma^{-1},\omega)^{-1})i(-u(\gamma\lambda)(\omega))=\\
&=\overline{s}(\sigma(\gamma^{-1},\omega)^{-1})i(-u(\gamma)(\omega))\overline{s}(\sigma(\lambda^{-1},\gamma^{-1}\omega)^{-1})i(-u(\lambda)(\gamma^{-1}\omega))=\\
&=\widetilde{\sigma}(\gamma^{-1},\omega)^{-1}\widetilde{\sigma}(\lambda^{-1},\gamma^{-1}\omega)^{-1} \ ,
\end{align*}
and this proves that $\widetilde{\sigma}$ is a cocycle.

Suppose now that there exists a measurable cocycle $\widetilde{\sigma}:\Gamma \times \Omega \rightarrow \widetilde{H}$ which is a lift of $\sigma$. Notice that, for every $\gamma \in \Gamma$ and almost every $\omega \in \Omega$, both $\widetilde{\sigma}(\gamma,\omega)$ and $\overline{s}(\sigma(\gamma,\omega))$ are mapped to the same element $\sigma(\gamma,\omega)$ via the projection $p:\widetilde{H} \rightarrow H$. Thus they must differ by an integer.

Inspired by Equation \eqref{eq:lift:formula} we can define a function $u$ as the unique function such that
\begin{equation}\label{eq:lift:coboundary}
u:\Gamma \rightarrow \textup{L}^\infty(\Omega,\bbZ) \ , \ \  \widetilde{\sigma}(\gamma^{-1},\omega)^{-1}=\overline{s}(\sigma(\gamma^{-1},\omega)^{-1})i(-u(\gamma)(\omega))\ .
\end{equation}
We need to show that $\overline{\delta}^1u=\overC^2(\sigma)(\overline{\varepsilon})$. Given $\gamma,\lambda \in \Gamma$ and $\omega \in \Omega$ we have that
\begin{align*}
&\overline{s}(\sigma(\lambda^{-1}\gamma^{-1},\omega)^{-1})i(-u(\gamma \lambda)(\omega))=\widetilde{\sigma}(\lambda^{-1}\gamma^{-1},\omega)^{-1}=\\
=&\widetilde{\sigma}(\gamma^{-1},\omega)^{-1}\widetilde{\sigma}(\lambda^{-1},\gamma^{-1}\omega)^{-1}=\\
=&\overline{s}(\sigma(\gamma^{-1},\omega)^{-1})i(-u(\gamma)(\omega))\overline{s}(\sigma(\lambda^{-1},\gamma^{-1}\omega)^{-1})i(-u(\lambda)(\gamma^{-1}\omega))=\\
=&\overline{s}(\sigma(\lambda^{-1}\gamma^{-1},\omega)^{-1})i(\overC^2(\sigma)(\overline{\varepsilon})(\gamma,\lambda)(\omega))i(-u(\gamma)(\omega))i(-u(\lambda)(\gamma^{-1})(\omega)) \ ,
\end{align*}
where we exploited Equation \eqref{eq:measurable:cocycle} to move from the first line to the second one, we used the definition of $\overline{\varepsilon}$ to pass from the second line to the third one and we concluded using Equation \eqref{eq:euler:pullback:section}. Hence, rearranging the terms of the equation, we get $\overline{\delta}^1u=\overC^2(\sigma)(\overline{\varepsilon})$ and the statement is proved. 
\end{proof}

The next step of our investigation aims to understand what happens when the pullback of the \emph{bounded} Euler class vanishes. In order to do this, we need first to introduce the following 

\begin{deft}\label{def:equivariant:family}
Let $\Gamma$ be a countable group and let $(\Omega,\mu)$ a standard Borel probability $\Gamma$-space. Given a measurable cocycle $\sigma:\Gamma \times \Omega \rightarrow H$ we say that it admits an \emph{equivariant family of points} if there exists a measurable map $r:\Omega \rightarrow \bbS^1$ such that
$$
r(\gamma \omega)=\sigma(\gamma,\omega)r(\omega) \ ,
$$
for every $\gamma \in \Gamma$ and almost every $\omega \in \Omega$. In this case we say that $r$ is a $\sigma$-\emph{equivariant} measurable function. 

Similarly we are going to say that a measurable cocycle $\widetilde{\sigma}:\Gamma \times \Omega \rightarrow \widetilde{H}$ admits an \emph{equivariant family of points} if  there exists a measurable function $\widetilde{r}:\Omega \rightarrow \bbR$ which is $\widetilde{\sigma}$-equivariant. We say that such a family is \emph{bounded} if the essential image of $\widetilde{r}$ is bounded. 
\end{deft}

\begin{lem}\label{lem:equivariant:family}
Let $\Gamma$ be a countable group and let $(\Omega,\mu)$ be a standard Borel probability $\Gamma$-space. Given a measurable cocycle $\sigma:\Gamma \times \Omega \rightarrow H$, suppose that $\textup{H}_b^2(\sigma)(e^b_{\bbZ})=0$. Then there exists a measurable cocycle $\widetilde{\sigma}:\Gamma \times \Omega \rightarrow \widetilde{H}$ which is a lift of $\sigma$ and $\widetilde{\sigma}$ admits a bounded equivariant family of points. 
\end{lem}

\begin{proof}
Since $\textup{H}^2_b(\sigma)(e^b_{\bbZ})=0$, the same holds in particular for $\textup{H}^2(\sigma)(e_{\bbZ})$. By Lemma \ref{lem:cocycle:lift}, the vanishing of $\textup{H}^2(\sigma)(e_{\bbZ})$ implies the existence of measurable cocycle $\widetilde{\sigma}:\Gamma \times \Omega \rightarrow \widetilde{H}$ which is a lift of $\sigma$. We are left to show that $\widetilde{\sigma}$ admits an equivariant family of points. 

Let $u:\Gamma \rightarrow \textup{L}^\infty(\Omega,\bbZ)$ be the function which satisfies
$$
\overC^2_b(\sigma)(\overline{\varepsilon})=\overline{\delta}^1u \ . 
$$
Since we assumed that the pullback of $e^b_{\bbZ}$ vanishes as a bounded cocycle, we know that $u$ is a bounded function, that is $|u(\gamma)(\omega)|$ is uniformly bounded for every $\gamma \in \Gamma$ and almost every $\omega \in \Omega$. 

From Equation \eqref{eq:lift:formula} we know that 
$$
\widetilde{\sigma}(\gamma^{-1},\omega)^{-1}=\overline{s}(\sigma(\gamma^{-1},\omega)^{-1})i(-u(\gamma)(\omega)) \ .
$$
By definition we know that $\overline{s}(\sigma(\gamma^{-1},\omega)^{-1})(0) \in [0,1)$. The boundedness of $u$ implies that the measurable function 
$$
\widetilde{r}:\Omega \rightarrow \bbR \ , \ \ \ \widetilde{r}(\omega):=\sup_{\gamma \in \Gamma} \widetilde{\sigma}(\gamma^{-1},\omega)^{-1}(0) \ ,
$$
is well-defined for almost every $\omega \in \Omega$. Additionally it is essentially bounded, still by the essential boundedness of $u$. We claim that this is the desired equivariant family of points. Clearly $\widetilde{r}$ is measurable. We need to show that is it $\widetilde{\sigma}$-equivariant. 
\begin{align*}
\widetilde{r}(\lambda^{-1} \omega)&= \sup_{\gamma \in \Gamma} \widetilde{\sigma}(\gamma^{-1},\lambda^{-1}\omega)^{-1}(0) =\\
&= \widetilde{\sigma}(\lambda^{-1},\omega) (\sup_{\gamma \in \Gamma}\widetilde{\sigma}(\gamma^{-1}\lambda^{-1},\omega)^{-1}(0))=\\
&=\widetilde{\sigma}(\lambda^{-1},\omega)\widetilde{r}(\omega) \ ,
\end{align*}
where we used Equation \eqref{eq:measurable:cocycle} to move from the first line to the second one. This concludes the proof. 
\end{proof}

In the case of representations there exists a well-defined notion of semiconjugacy which is intimately related with the pullback of the Euler class. Here we want to introduce a notion of \emph{semicohomology} which will be the extension of semiconjugacy to the world of measurable cocycles. The notion of semicohomology appeared in the work of Bader, Furman and Shaker \cite{BFS06}. Here we are going to follow the line of Bucher, Frigerio and Hartnick \cite{BFH} since our intertwining maps will not be continuous.  We first recall that a \emph{non-decreasing degree one map} $\varphi:\bbS^1 \rightarrow \bbS^1$  is a weakly order preserving map for every $k$-tuples of points on the circle and for every $k \in \bbN$. 

\begin{deft}\label{def:semicohomology}
Let $\Gamma$ be a countable group and let $(\Omega,\mu)$ be a standard Borel probability $\Gamma$-space. We say that $\sigma_1$ is \emph{left semicohomologous} to $\sigma_2$ if there exists a measurable map $\varphi:\bbS^1 \times \Omega \rightarrow \bbS^1$ such that the slice $\varphi(\omega):\bbS^1 \rightarrow \bbS^1\ , \ \ (\varphi(\omega))(\xi):=\varphi(\xi,\omega)$ is a non-decreasing degree one map for almost every $\omega \in \Omega$ and it holds
$$
\sigma_1(\gamma,\omega)\varphi(\omega)=\varphi(\gamma \omega)\sigma_2(\gamma,\omega) \ ,
$$
for every $\gamma \in \Gamma$ and almost every $\omega \in \Omega$. We will equivalently say that $\sigma_2$ is \emph{right semicohomologous} to $\sigma_1$. 

We say that $\sigma_1$ is \emph{semicohomologous} to $\sigma_2$ if it is both left and right semicohomologous to $\sigma_2$. Finally, we say that $\sigma_1$ and $\sigma_2$ are \emph{cohomologous} if $\varphi(\omega) \in H$ for almost every $\omega \in \Omega$. 
\end{deft}

\begin{prop}\label{prop:semicohomology:trivial}
Let $\Gamma$ be a countable group and let $(\Omega,\mu)$ be a standard Borel probability $\Gamma$-space. Let $\sigma:\Gamma \times \Omega \rightarrow H$ be a measurable cocycle. Then $\sigma$ is right semicohomologous to the trivial cocycle. It is semicohomologous to the trivial cocycle if and only if it admits an equivariant family of points. 
\end{prop}

\begin{proof}
To show that $\sigma:\Gamma \times \Omega \rightarrow H$ is right semicohomologous to the trivial cocycle (that is the one induced by the trivial representation) it is sufficient to consider $\varphi:\bbS^1 \times \Omega \rightarrow \bbS^1$ defined by $\varphi(x,\omega)=x_0$ for almost every $x \in \bbS^1, \omega \in \Omega$. Here $x_0 \in \bbS^1$ is an arbitrary point. Indeed we have
$$
(\varphi(\gamma \omega))(\sigma(\gamma,\omega)(x))=x_0=(\varphi(\omega))(x) \ ,
$$
for every $\gamma \in \Gamma$ and almost every $x \in \bbS^1, \omega \in \Omega$. 

Suppose now that $\sigma$ is semicohomologous to the trivial cocycle. As a consequence there exists a measurable map $\varphi:\bbS^1 \times \Omega \rightarrow \bbS^1$, which is a left semicohomology between $\sigma$ and the trivial cocycle. More precisely, we must have
\begin{equation}\label{eq:left:semicohomology:trivial}
\sigma(\gamma,\omega)\varphi(\omega)=\varphi(\gamma \omega) \ ,
\end{equation}
for every $\gamma \in \Gamma$ and almost every $\omega \in \Omega$. Here $\varphi(\omega):\bbS^1 \rightarrow \bbS^1$ is the $\omega$-slice associated to $\varphi$. 

To get an equivariant family of points for $\sigma$ it is sufficient to define the measurable function 
$$
r:\Omega \rightarrow \bbS^1 \ , \ \ \ r(\omega):=\varphi(\omega)(x_0) \ ,
$$
for some $x_0 \in \bbS^1$. Exploiting Equation \eqref{eq:left:semicohomology:trivial} we obtain immediately that
$$
r(\gamma \omega)=\varphi(\gamma \omega)(x_0)=\sigma(\gamma,\omega)(\varphi(\omega)(x_0))=\sigma(\gamma,\omega)r(\omega) \ ,
$$
that is $r$ is $\sigma$-equivariant and the claim is proved. 

Suppose now that $\sigma$ admits an equivariant family of points. This implies the existence of a measurable map $r:\Omega \rightarrow \bbS^1$ which is $\sigma$-equivariant. 
To obtain a left semicohomology between $\sigma$ and the trivial cocycle it is sufficient to define
$$
\varphi:\bbS^1 \times \Omega \rightarrow \bbS^1 \ , \ \ \ \varphi(x,\omega):=r(\omega) \ . 
$$
Clearly $r(\omega)$ is a non-decreasing degree one map and it holds
$$
\sigma(\gamma,\omega)(\varphi(\omega)(x))=\sigma(\gamma,\omega)r(\omega)=r(\gamma \omega)=\varphi(\gamma \omega)(x) \ ,
$$
for every $\gamma \in \Gamma, x \in \bbS^1$ and almost every $\omega \in \Omega$. This shows that $\sigma$ is left semicohomologous to the trivial cocycle and it concludes the proof. 
\end{proof}

We are finally ready to give a complete characterization of measurable cocycles whose parametrized Euler class vanishes identically. 

\begin{teor}\label{teor:vanishing:euler}
Let $\Gamma$ be a countable group and let $(\Omega,\mu)$ be a standard Borel probability $\Gamma$-space. Consider a measurable cocycle $\sigma:\Gamma \times \Omega \rightarrow H$. Then the following are equivalent.
\begin{enumerate}
	\item The parametrized Euler class vanishes, that is $\textup{H}^2_b(\sigma)(e^b_{\bbZ})=0$\ .
	\item There exists a measurable cocycle $\widetilde{\sigma}:\Gamma \times \Omega \rightarrow \widetilde{H}$ which is a lift of $\sigma$ and $\widetilde{\sigma}$ admits a bounded equivariant family of points in $\bbR$\ .
	\item There exists a $\sigma$-equivariant family of points in $\bbS^1$.
	\item The cocycle is semicohomologous to the trivial cocycle\ .
\end{enumerate}
\end{teor}

\begin{proof}
The fact that ($1$) implies ($2$) is the content of Lemma \ref{lem:equivariant:family}. We need to show that ($2$) implies ($1$). Since there exists a lift $\widetilde{\sigma}:\Gamma \times \Omega \rightarrow \widetilde{H}$, by Lemma \ref{lem:cocycle:lift} there exists a function $u:\Gamma \rightarrow \textup{L}^\infty(\Omega,\bbZ)$ such that $\overline{\delta}^1 u=\overC^2(\sigma)(\overline{\varepsilon})$. An explicit expression of the function $u$ is given by
$$
\widetilde{\sigma}(\gamma^{-1},\omega)^{-1}=\overline{s}(\sigma(\gamma^{-1},\omega)^{-1})i(-u(\gamma)(\omega)) \ . 
$$ 
Let $\widetilde{r}:\Omega \rightarrow \bbR$ be the bounded equivariant family associated to $\widetilde{\sigma}$. Then 
\begin{align*}
u(\gamma)(\omega)&=\overline{s}(\sigma(\gamma^{-1},\omega)^{-1})(\widetilde{r}(\gamma^{-1}\omega))-\widetilde{\sigma}(\gamma^{-1},\omega)^{-1}(\widetilde{r}(\gamma^{-1}\omega))=\\
&=\overline{s}(\sigma(\gamma^{-1},\omega)^{-1})(\{\widetilde{r}(\gamma^{-1}\omega)\})+\lfloor \widetilde{r}(\gamma^{-1}\omega) \rfloor -\widetilde{r}(\omega) \ ,
\end{align*}
where $\{\widetilde{r}(\gamma^{-1}\omega)\}$ and $\lfloor \widetilde{r}(\gamma^{-1}\omega) \rfloor$ are the fractionary and the integer part of $\widetilde{r}(\gamma^{-1}\omega)$, respectively. Since $\overline{s}(\sigma(\gamma^{-1},\omega)^{-1})(x)$ lies inside the interval $[0,2)$ if $x \in [0,1)$, the essential boundedness of $\widetilde{r}$ implies the boundedness of $u$ with respect to every $\gamma \in \Gamma$ and almost every $\omega \in \Omega$. Thus the claim is proved and we have shown that ($1$) is equivalent to ($2$). 

Suppose that there exists a measurable cocycle $\widetilde{\sigma}$ which is a lift of $\sigma$ admitting a bounded equivariant family of points in $\bbR$. Let $\widetilde{r}:\Omega \rightarrow \bbR$ be such a family. To obtain an equivariant family of points for $\sigma$ it is sufficient to consider the composition 
$$
r:\Omega \rightarrow \bbS^1 \ , \ \ \ r(\omega)=\pi(\widetilde{r}(\omega)) \ ,
$$
where $\pi:\bbR \rightarrow \bbS^1$ is the natural projection. By the $\widetilde{\sigma}$-equivariance of $\widetilde{r}$ we get 
$$
r(\gamma \omega)=\pi(\widetilde{r}(\gamma \omega))=\pi(\widetilde{\sigma}(\gamma,\omega)\widetilde{r}(\omega))=\sigma(\gamma,\omega)\pi(\widetilde{r}(\omega))=\sigma(\gamma,\omega)r(\omega) \ ,
$$
for every $\gamma \in \Gamma$ and almost every $\omega \in \Omega$. 

Viceversa, suppose that there exists a $\sigma$-equivariant family of points in $\bbS^1$. Let $r:\Omega \rightarrow \bbS^1$ be such a family. Since the projection $\pi:\bbR \rightarrow \bbS^1$ admits a measurable section $s:\bbS^1 \rightarrow \bbR$, we can consider the composition 
$$
\widetilde{r}:\Omega \rightarrow \bbR \ , \ \ \ \widetilde{r}(\omega):=s(r(\omega)) \ ,
$$
for almost every $\omega \in \Omega$. Clearly $\widetilde{r}$ is measurable, being the composition of two measurable functions. Additionally if we suppose that $s$ has image contained in $[0,1)$, then $\widetilde{r}$ is essentially bounded. 

Given $\gamma \in \Gamma$, since $r$ is $\sigma$-equivariant, the element $\sigma(\gamma,\omega)$ restricts to a homeomorphism
$$
\sigma(\gamma,\omega): \bbS^1 \setminus \{r(\omega)\} \rightarrow \bbS^1 \setminus \{r(\gamma \omega)\} \ . 
$$
We can lift the above map to obtain a map 
$$
\widetilde{\sigma}(\gamma,\omega):(\widetilde{r}(\omega),\widetilde{r}(\omega)+1) \rightarrow (\widetilde{r}(\gamma \omega),\widetilde{r}(\gamma \omega)+1) \ ,
$$
which is a homeomorphism. By periodicity we can extend such homeomorphism to 
$$
\widetilde{\sigma}(\gamma,\omega): \bigsqcup_{n \in \bbZ} (\widetilde{r}(\omega)+n,\widetilde{r}(\omega)+n+1) \rightarrow \bigsqcup_{n \in \bbZ} (\widetilde{r}(\gamma \omega)+n,\widetilde{r}(\gamma \omega)+n+1) \ ,
$$
Finally we can extend it to the whole real line by setting
$$
\widetilde{\sigma}(\gamma,\omega)(\widetilde{r}(\omega)+n):=\widetilde{r}(\gamma \omega)+n \, 
$$
for every $n \in \bbZ$. In this way we obtain a measurable map $\widetilde{\sigma}:\Gamma \times \Omega \rightarrow \widetilde{H}$, which is a lift of $\sigma$. Additionally $\widetilde{\sigma}$ is a cocycle since both $\widetilde{\sigma}(\gamma \lambda,\omega)\widetilde{r}(\omega)$ and $(\widetilde{\sigma}(\gamma,\lambda\omega)\widetilde{\sigma}(\lambda,\omega))(\widetilde{r}(\omega))$ are equal to $\widetilde{r}(\gamma \lambda \omega)$. Finally $\widetilde{r}$ is a bounded equivariant family for $\widetilde{\sigma}$. This concludes the equivalence of ($2$) and ($3$).

The equivalence between ($3$) and ($4$) is the content of Proposition \ref{prop:semicohomology:trivial}. 
\end{proof}

Using the previous theorem we can prove ($1$) of Theorem \ref{teor:main:intro}.

\begin{teor}\label{teor:ghys:vanishing}
Let $\Gamma$ be a countable group and let $(\Omega,\mu)$ be a standard Borel probability $\Gamma$-space. Let $\sigma_1,\sigma_2:\Gamma \times \Omega \rightarrow H$ be two measurable cocycles which are semicohomologous. If one of them, say $\sigma_1$, has an equivariant family of points, the same holds for $\sigma_2$ and we have
$$
\textup{H}^2_b(\sigma_1)(e^b_{\bbZ})=\textup{H}^2_b(\sigma_2)(e^b_{\bbZ})=0 \ .
$$
\end{teor}

\begin{proof}
This is a direct consequence of Theorem \ref{teor:vanishing:euler}. Indeed if $\sigma_1$ has an equivariant family of points, then $\sigma_1$ is semicohomologous to the trivial action. By the fact that semicohomology is an equivalence relation, we get that $\sigma_2$ is semicohomologous to the trivial cocycle as well. Then its admits an equivariant family of points and for both $\sigma_1$ and $\sigma_2$ the parametrized Euler class must vanish. 
\end{proof}

\section{Ghys' theorem for parametrized Euler classes}\label{sec:ghys:theorem}

In this section we are going to study the relation between the semicohomology class of a measurable cocycle and the parametrized Euler class. We are going to see that the parametrized Euler class is a complete invariant for the semicohomology class. 

We start by showing ($2$) of Theorem \ref{teor:main:intro}. The following is a generalization of \cite[Theorem A]{ghys:articolo} to the world of measurable cocycles. 

\begin{teor}\label{teor:ghys:pullback:semicohomology}
Let $\Gamma$ be a countable group and let $(\Omega,\mu)$ be a standard Borel probability $\Gamma$-space. Let $\sigma_1,\sigma_2:\Gamma \times \Omega \rightarrow H$ be two measurable cocycles. Suppose that 
$$
\textup{H}^2_b(\sigma_1)(e^b_{\bbZ})=\textup{H}^2_b(\sigma_2)(e^b_{\bbZ}) \ .
$$
Then $\sigma_1$ and $\sigma_2$ are semicohomologous. 
\end{teor}

\begin{proof}
By hypothesis, we know that 
$$
\textup{H}^2_b(\sigma_1)(e^b_{\bbZ})=\textup{H}^2_b(\sigma_2)(e^b_{\bbZ})\ .
$$
Thus there exists a bounded function $u:\Gamma \rightarrow \textup{L}^\infty(\Omega,\bbZ)$ such that 
$$
\overC_b^2(\sigma_1)(\overline{\varepsilon})-\overC_b^2(\sigma_2)(\overline{\varepsilon})=\overline{\delta}^1u \ .
$$
Recall that $\textup{L}^\infty(\Omega,\bbZ)$ has a natural structure of $\Gamma$-module with the $\Gamma$-action given by
$$
(\gamma f)(\omega):=f(\gamma^{-1} \omega) \ .
$$

As a direct consequence of Section \ref{sec:extension}, there exists a group $\widetilde{\Gamma}$  which is an extension of $\Gamma$ by $\textup{L}^\infty(\Omega,\bbZ)$, that is  it fits into a short exact sequence of groups
$$
\xymatrix{
0 \ar[r] & \textup{L}^\infty(\Omega,\bbZ) \ar[r]^{\ \ \ \ j} & \widetilde{\Gamma} \ar[r] & \Gamma \ar[r] & 0 \\
}
$$
and there exist two sections 
$$
s_1,s_2:\Gamma \rightarrow \widetilde{\Gamma} \ ,
$$
such that $s_2(\gamma)=s_1(\gamma)j(u(\gamma))$ for every $\gamma \in \Gamma$. Additionally, for $k=1,2$, the section $s_k$, together with the $\Gamma$-module structure on $\upL^\infty(\Omega,\bbZ)$, determines the group law on $\widetilde{\Gamma}$. Indeed given $\gamma,\lambda \in \Gamma$ and $f,g \in \textup{L}^\infty(\Omega,\bbZ)$ it holds
\begin{equation}\label{eq:group:law}
\left(s_k(\lambda^{-1})j(f)\right)\left(s_k(\gamma^{-1})j(g)\right)=s_k(\lambda^{-1} \gamma^{-1})j(\overC^2_b(\sigma_k)(\overline{\varepsilon})(\gamma,\lambda)+(\gamma f)+g) \ .
\end{equation}
The presence of the inverses of both $\gamma$ and $\lambda$ in the above formula is due to the definition of the pullback $\overC^2_b(\sigma_k)(\overline{\varepsilon})$. 

Notice that we can endow $(\Omega,\mu)$ with a natural $\widetilde{\Gamma}$ action which factors through the given $\Gamma$-action, that is 
$$
\left(s_k(\gamma)j(f)\right)\omega:=\gamma \omega \ ,
$$
for every $\gamma \in \Gamma, f \in \textup{L}^\infty(\Omega,\bbZ)$ and $\omega \in \Omega$. In this way $(\Omega,\mu)$ becomes a standard Borel probability $\widetilde{\Gamma}$-space.

For $k=1,2$ fixed, our first goal is to construct a measurable cocycle $\widetilde{\sigma}_k:\widetilde{\Gamma} \times \Omega \rightarrow \widetilde{H}$ such that the following diagram commutes
$$
\xymatrix{
0 \ar[r] & \textup{L}^0(\Omega,\bbZ) \ar[r]_{i_\Omega} & \textup{L}^0(\Omega,\widetilde{H}) \ar[r]_{p_\Omega} & \textup{L}^0(\Omega,H) \ar[r] & 0\\
0 \ar[r] & \textup{L}^\infty(\Omega,\bbZ) \ar[u] \ar[r]_j & \widetilde{\Gamma} \ar[u]^{\widetilde{\sigma}_k} \ar[r] & \Gamma \ar[u]^{\sigma_k} \ar[r] & 0 \\
}
$$
We define
$$
\widetilde{\sigma}_k(s_k(\gamma^{-1})j(f),\omega)^{-1}:=\overline{s}(\sigma_k(\gamma^{-1},\omega)^{-1})i(f(\omega)) \ , 
$$
for every $\gamma \in \Gamma, f \in \textup{L}^\infty(\Omega,\bbZ)$ and almost every $\omega \in \Omega$. We claim that $\widetilde{\sigma}_k$ satisfies Equation \eqref{eq:measurable:cocycle}. Let $\gamma,\lambda \in \Gamma$ and let $f,g \in \textup{L}^\infty(\Omega,\bbZ)$. It holds that
\begin{align*}
&\widetilde{\sigma}_k(s_k(\lambda^{-1})j(f)s_k(\gamma^{-1})j(g),\omega)^{-1}=\\
=&\widetilde{\sigma}_k(s_k(\lambda^{-1} \gamma^{-1} )j(\overC_b^2(\sigma_k)(\overline{\varepsilon})(\gamma,\lambda)+(\gamma f)+g), \omega)^{-1}=\\
=&\overline{s}(\sigma_k(\lambda^{-1} \gamma^{-1},\omega)^{-1})i(\overC_b^2(\sigma_k)(\overline{\varepsilon})(\gamma,\lambda)(\omega)+f(\gamma^{-1} \omega)+g(\omega))=\\
=&\overline{s}(\sigma_k(\gamma^{-1}, \omega)^{-1})\overline{s}(\sigma_k(\lambda^{-1},\gamma^{-1} \omega)^{-1})i(f(\gamma^{-1} \omega)+g(\omega))=\\
=&\overline{s}(\sigma_k(\gamma^{-1},\omega)^{-1})i(g(\omega))\overline{s}(\sigma_k(\lambda^{-1},\gamma^{-1}\omega)^{-1})i(f(\gamma^{-1}\omega))=\\
=&\widetilde{\sigma}_k(s_k(\gamma^{-1})j(g),\omega)^{-1}\widetilde{\sigma}(s_k(\lambda^{-1})j(f),(s_k(\gamma^{-1})j(g)) \cdot\omega)^{-1} \ ,
\end{align*}
where we used Equation \eqref{eq:group:law} in the first line, then we applied the definition of $\widetilde{\sigma}_k$ to move from the first line to the second one, we exploited the definition of $\overC_b^2(\sigma)(\overline{\varepsilon})$ to pass from the second line to the third one and we concluded thanks to the fact that $i(\bbZ)$ is central in $\widetilde{H}$. Hence $\widetilde{\sigma}_k$ satisfies Equation \eqref{eq:measurable:cocycle}. Additionally such extension is measurable, since it is defined in terms of measurable functions. Thus $\widetilde{\sigma}_k$ is a measurable cocycle and the claim is proved. 

The next step in the proof is to show that, for any fixed $x \in \bbR$, the quantity 
$$\widetilde{\varphi}(\alpha,\omega)(x):=\widetilde{\sigma}_1(\alpha,\omega)^{-1}\widetilde{\sigma}_2(\alpha,\omega)(x)$$
is bounded for every $\alpha \in \widetilde{\Gamma}$ and almost every $\omega \in \Omega$. Since $\widetilde{\varphi}$ is defined is terms of increasing homeomorphisms that commute with integer translations, it is sufficient to check the boundedness only at $x=0$. 

We first notice that $\widetilde{\varphi}(\alpha,\omega)$ depends only on the projection of $\alpha$ on the group $\Gamma$. More precisely, consider $\beta=\alpha j(f)$ for some $f \in \textup{L}^\infty(\Omega,\bbZ)$. Then it follows that
\begin{align*}
\widetilde{\varphi}(\beta,\omega)&=\widetilde{\sigma}_1(\alpha j(f),\omega)^{-1}\widetilde{\sigma}_2(\alpha j(f),\omega)=\\
&=i(-f(\omega))\widetilde{\sigma}_1(\alpha,\omega)^{-1}\widetilde{\sigma}_2(\alpha,\omega)i(f(\omega))=\\
&=\widetilde{\sigma}_1(\alpha,\omega)^{-1}\widetilde{\sigma}_2(\alpha,\omega)=\widetilde{\varphi}(\alpha,\omega) \ ,
\end{align*}
where we used the fact that $\widetilde{\sigma}_1$ and $\widetilde{\sigma}_2$ are cocycles and that $j(\textup{L}^\infty(\Omega,\bbZ))$ acts trivially on $\Omega$ to move from the first line to the second one and we concluded exploiting the fact $i(\bbZ)$ is central in $\widetilde{H}$. Thus the function $\widetilde{\varphi}(\alpha,\omega)$ depends only on the projection of $\widetilde{\Gamma}$ on $\Gamma$.

Consider now an element $\gamma \in \Gamma$ and fix $\alpha=s_2(\gamma^{-1})$, where $s_2$ is one of the two sections introduced at the beginning of the proof. By evaluating $\widetilde{\sigma}_2$ on $\alpha$ we get
$$
\widetilde{\sigma}_2(\alpha,\omega)(0)=\overline{s}(\sigma(\gamma^{-1},\omega)^{-1})^{-1}(0) \in (-1,0] \ .
$$
If we now plug in the relation $s_2(\gamma^{-1})=s_1(\gamma^{-1})j(u(\gamma^{-1}))$ into the cocycle $\widetilde{\sigma}_1$ we obtain 
\begin{align*}
\widetilde{\sigma}_1(\alpha,\omega)^{-1}&=\widetilde{\sigma}_1(s_2(\gamma^{-1}),\omega)^{-1}=\\
&=\widetilde{\sigma}_1(s_1(\gamma^{-1})j(u(\gamma^{-1})),\omega)^{-1}=\overline{s}(\sigma_1(\gamma^{-1},\omega)^{-1})i(u(\gamma^{-1})(\omega)) \ .
\end{align*}

Recall that for any $x \in [-1,1)$, we have that $\overline{s}(\sigma_1(\gamma^{-1},\omega)^{-1})(x) \in [-1,2)$. As a consequence 
$$
\widetilde{\sigma}_1(\alpha,\omega)^{-1}(\overline{s}(\sigma(\gamma^{-1},\omega)^{-1})^{-1}(0)) \in [u(\gamma^{-1})(\omega)-1,u(\gamma^{-1})(\omega)+2) \ , 
$$
The boundedness of $u$ implies that the above interval is bounded, thus the quantity $\widetilde{\varphi}(\alpha,\omega)(0)$ is bounded for every $\alpha \in \widetilde{\Gamma}$ and almost every $\omega \in \Omega$. 

As a result we have a well-defined map
$$
\widetilde{\varphi}:\bbR \times \Omega \rightarrow \bbR \ , \ \ \ \widetilde{\varphi}(x,\omega):=\sup_{\alpha \in \widetilde{\Gamma}} \widetilde{\varphi}(\alpha,\omega)(x) \ .
$$
%
%
Being defined in terms of measurable functions, the function $\widetilde{\varphi}$ is measurable. Moreover, for almost every $\omega$, the slice $\widetilde{\varphi}(\omega)(x):=\widetilde{\varphi}(x,\omega)$ is a non-decreasing map commuting with integer translations. Indeed it is the supremum of increasing homeomorphisms commuting with integer translations. 

For any $\beta \in \widetilde{\Gamma}, x \in \bbR$ and almost every $\omega \in \Omega$ we have that 
\begin{align}\label{eq:lift:semicohomology}
\left(\widetilde{\varphi}(\beta \omega)\widetilde{\sigma}_2(\beta,\omega)\right)(x)&=\sup_{\alpha \in \widetilde{\Gamma}} \left(\widetilde{\sigma}_1(\alpha,\beta \omega)^{-1}\widetilde{\sigma}_2(\alpha,\beta\omega)\widetilde{\sigma}_2(\beta,\omega)\right)(x)=\\
&=\sup_{\alpha \in \widetilde{\Gamma}} \left(\widetilde{\sigma}_1(\alpha,\beta \omega)^{-1}\widetilde{\sigma}_2(\alpha\beta,\omega)\right)(x)= \nonumber\\
&=\sup_{\theta \in \widetilde{\Gamma}} \left(\widetilde{\sigma}_1(\theta \beta^{-1},\beta \omega)^{-1}\widetilde{\sigma}_2(\theta,\omega)\right)(x)= \nonumber\\
&=\sup_{\theta \in \widetilde{\Gamma}} \left(\widetilde{\sigma}_1(\beta^{-1},\beta \omega)^{-1}\widetilde{\sigma}_1(\theta,\omega)^{-1}\widetilde{\sigma}_2(\theta,\omega)\right)(x)= \nonumber \\
&=\left(\widetilde{\sigma}_1(\beta,\omega) \widetilde{\varphi}(\omega)\right)(x) \ . \nonumber
\end{align}
In the above computation we used Equation \eqref{eq:measurable:cocycle} to move from the first line to the second one, then we changed variable from $\alpha$ to $\theta=\alpha\beta$ in the third line and we exploited again Equation \eqref{eq:measurable:cocycle} to conclude. 

Since $\widetilde{\varphi}(\omega)$ is a non-decreasing map commuting with integer translations, it descends to a non-decreasing degree one map $\varphi(\omega):\bbS^1 \rightarrow \bbS^1$ by \cite[Lemma 2.4]{BFH}. Moreover, as a consequence of Equation \eqref{eq:lift:semicohomology}, it follows that 
$$
\varphi(\gamma \omega)\sigma_2(\gamma,\omega)=\sigma_1(\gamma,\omega)\varphi(\omega) \ ,
$$
that is $\sigma_1$ is left semicohomologous to $\sigma_2$. To find a right semicohomology is sufficient to apply the same reasoning to $\widetilde{\varphi}^\ast(\alpha,\omega)=\sup_{\alpha \in \widetilde{\Gamma}}\widetilde{\sigma}_2(\alpha,\omega)^{-1}\widetilde{\sigma}_1(\alpha,\omega)$. Thus $\sigma_1$ is both right and left semicohomologous to $\sigma_2$. This finishes the proof. 
\end{proof}

To conclude the proof of our Theorem \ref{teor:main:intro}, we are left to show that if two measurable cocycles are semicohomologous, then they determine the same parametrized Euler class. We will actually need to assume that the space $(\Omega,\mu)$ is $\Gamma$-ergodic in order to get the statement.

\begin{prop}\label{prop:ghys:nonvanishing}
Let $\Gamma$ be a countable group and let $(\Omega,\mu)$ be an ergodic standard Borel probability $\Gamma$-space. Given $\sigma_1,\sigma_2:\Gamma \times \Omega \rightarrow H$ two measurable cocycles, assume that $\sigma_1$ is left semicohomologous to $\sigma_2$ and that $\sigma_1$ does not admit any equivariant family of points. Then it holds 
$$
\textup{H}^2_b(\sigma_1)(e^b_{\bbZ})=\textup{H}^2_b(\sigma_2)(e^b_{\bbZ}) \neq 0 \ .
$$
\end{prop}

\begin{proof}
Since $\sigma_1$ is left semicohomologous to $\sigma_2$ there exists a measurable function $\varphi:\bbS^1 \times \Omega \rightarrow \bbS^1$ such that 
\begin{equation}\label{eq:semicohomology:local:left}
\varphi(\gamma \omega)\sigma_2(\gamma,\omega)=\sigma_1(\gamma,\omega)\varphi(\omega) \ ,
\end{equation}
for every $\gamma \in \Gamma$ and almost every $\omega \in \Omega$. We claim that $\varphi(\omega)$ does not coincide with the constant map for almost every $\omega \in \Omega$. Define
$$
A_0: = \{\omega \in \Omega \ | \ \textup{$\varphi(\omega)$ is constant} \} \ .
$$
The above set is measurable by the measurability of $\varphi$. By Equation \eqref{eq:semicohomology:local:left} $\varphi(\omega)$ is constant if and only if $\varphi(\gamma \omega)$ is constant. This precisely means that $A_0$ is a $\Gamma$-invariant set. By the ergodicity assumption $A_0$ has either full or null measure. 

By contradiction suppose that $A_0$ has full measure. Then for almost every $\omega \in \Omega$ the map $\varphi(\omega)$ is constant. For $x_0 \in \bbS^1$ fixed, we define
$$
r:\Omega \rightarrow \bbS^1 \ , \ \ r(\omega):=\varphi(\omega)(x_0) \ .
$$
As a consequence of Equation \eqref{eq:semicohomology:local:left} it follows that $r$ is a $\sigma_1$-equivariant family of points. Since this is in contradiction with our hypothesis, then $A_0$ must be of null measure. Hence for almost every $\omega \in \Omega$ the function $\varphi(\omega)$ cannot be constant. 

For a given $\omega \in \Omega$, we can consider a lift $\widetilde{\varphi}(\omega):\bbR \rightarrow \bbR$ of $\varphi(\omega)$ so that $\widetilde{\varphi}(\omega)(0) \in [0,1)$. We denote by $\overline{s}:H \rightarrow \widetilde{H}$ the usual section $\overline{s}(f)(0) \in [0,1)$. Then there exists a function $u:\Gamma \times \Omega \times \bbR \rightarrow \bbZ$ such that
$$
\left(\overline{s}(\sigma_1(\gamma^{-1},\omega)^{-1})\widetilde{\varphi}(\gamma^{-1} \omega)\right)(x)=\left(\widetilde{\varphi}(\omega)\overline{s}(\sigma_2(\gamma^{-1},\omega)^{-1}\right)(x)+u(\gamma,\omega)(x) \ ,
$$
for every $\gamma \in \Gamma, x \in \bbR$ and almost every $\omega \in \Omega$. The (temporary) dependence of $u$ on the variable $x \in \bbZ$ is due to the fact that we are not dealing with homeomorphisms anymore, but only with lifts of suitable functions which coincide pointwise on the circle. 

We claim that the function $u$ is actually independent of the real variable $x$. To show this, recall that $\varphi(\gamma^{-1} \omega)$ is not constant for almost every $\omega \in \Omega$. Thus we can find $t,s  \in (0,1)$ such that $\widetilde{\varphi}(\gamma^{-1}\omega)(t)-\widetilde{\varphi}(\gamma^{-1}\omega)(s) \in (0,1)$, for almost every $\omega \in \Omega$ (the points $t,s$ depend both on $\gamma$ and $\omega$ but we do not want make heavier an already heavy notation). Since $\overline{s}(\sigma_1(\gamma^{-1},\omega)^{-1})$ is a strictly increasing homeomorphism, it must hold
\begin{equation}\label{eq:sandwich:zeroone}
0< \overline{s}(\sigma_1(\gamma^{-1},\omega)^{-1})\widetilde{\varphi}(\gamma^{-1} \omega)(t)-\overline{s}(\sigma_1(\gamma^{-1},\omega)^{-1})\widetilde{\varphi}(\gamma^{-1} \omega)(s)< 1 \ .
\end{equation}
Similarly, since $\widetilde{\varphi}(\omega)\overline{s}(\sigma_2(\gamma^{-1},\omega)^{-1})$ is non-decreasing, we get that 
$$
0 \leq \widetilde{\varphi}(\omega)\overline{s}(\sigma_2(\gamma^{-1},\omega)^{-1})(t)-\widetilde{\varphi}(\omega)\overline{s}(\sigma_2(\gamma^{-1},\omega)^{-1})(s) \leq 1 \ .
$$
By the fact that $\overline{s}(\sigma_1(\gamma^{-1},\omega)^{-1})\widetilde{\varphi}(\gamma^{-1} \omega)$ and $\widetilde{\varphi}(\omega)\overline{s}(\sigma_2(\gamma^{-1},\omega)^{-1})$ have the same projection, the above inequality must be strict, otherwise Equation \eqref{eq:sandwich:zeroone} would be violated. Hence it must hold 
\begin{equation}\label{eq:zeroone:sandwich}
0 < \widetilde{\varphi}(\omega)\overline{s}(\sigma_2(\gamma^{-1},\omega)^{-1})(t)-\widetilde{\varphi}(\omega)\overline{s}(\sigma_2(\gamma^{-1},\omega)^{-1})(s) < 1 \ .
\end{equation}
Subtracting Equation \eqref{eq:sandwich:zeroone} from Equation \eqref{eq:zeroone:sandwich} we get that 
$$
u(\gamma,\omega)(t)-u(\gamma,\omega)(s) \in (-1,1) \ ,
$$
for every $\gamma \in \Omega$ and almost every $\omega \in \Omega$. Since $u$ is integer-valued, we must have that 
$$
u(\gamma,\omega)(t)=u(\gamma,\omega)(s) \ .
$$
In this way we proved that the function $u(\gamma,\omega)$ is constant on the set
$$
E:=(t + \bbZ) \sqcup (s + \bbZ) \ .
$$
Let now $x \in \bbR \setminus E$. We know that there exist two points in $E \cap (x,x+1)$ which are translated of $t$ and $s$, respectively. The fact that $\widetilde{\varphi}(\gamma^{-1}\omega)$ is non-decreasing guarantees that we can find a point $e \in E$ such that 
$$
e -x \in (0,1) \ ,
$$
and similarly 
$$
\widetilde{\varphi}(\gamma^{-1}\omega)(e)-\widetilde{\varphi}(\gamma^{-1}\omega)(x) \in (0,1) \ .
$$
Following the same reasoning as above, we can conclude that 
$$
u(\gamma,\omega)(e)-u(\gamma,\omega)(x) \in (-1,1) \ ,
$$
hence $u(\gamma,\omega)$ is constant on $\bbR \setminus E$. Thus $u(\gamma,\omega)$ does not depend on $x \in \bbR$ and the claim is proved. Thus it follows that 
\begin{equation}\label{eq:semicohomology:inverse}
\overline{s}(\sigma_1(\gamma^{-1},\omega)^{-1})\widetilde{\varphi}(\gamma^{-1} \omega)=\widetilde{\varphi}(\omega)\overline{s}(\sigma_2(\gamma^{-1},\omega)^{-1})i(u(\gamma,\omega)) \ .
\end{equation}

Given $\gamma,\lambda \in \Gamma$, we start evaluating the left-hand side of Equation \eqref{eq:semicohomology:inverse} on the product $\lambda^{-1}\gamma^{-1}$, that is 
\begin{align*}
&\textup{LHS}=\\
=&\overline{s}(\sigma_1(\lambda^{-1}\gamma^{-1},\omega)^{-1})\widetilde{\varphi}(\lambda^{-1}\gamma^{-1}\omega)=\\
=&i(-\overC^2(\sigma_1)(\overline{\varepsilon})(\gamma,\lambda))\overline{s}(\sigma_1(\gamma^{-1},\omega)^{-1})\overline{s}(\sigma_1(\lambda^{-1},\gamma^{-1}\omega)^{-1})\widetilde{\varphi}(\lambda^{-1}\gamma^{-1}\omega)=\\
=&\overline{s}(\sigma_1(\gamma^{-1},\omega)^{-1})\widetilde{\varphi}(\gamma^{-1}\omega)\overline{s}(\sigma_2(\lambda^{-1},\gamma^{-1}\omega)^{-1})i(-\overC^2(\sigma_1)(\overline{\varepsilon})(\gamma,\lambda)(\omega)+u(\lambda)(\gamma^{-1}\omega)) \ 
,
\end{align*}
where we exploited the definition of $\overline{\varepsilon}$ to pass from the first line to the second one and we concluded exploting Equation \eqref{eq:semicohomology:inverse}. Similarly we can evaluate the right-hand side of Equation \eqref{eq:semicohomology:inverse} on the product $\lambda^{-1}\gamma^{-1}$ and we get that 
\begin{align*}
&\textup{RHS}=\\
=&\widetilde{\varphi}(\omega)\overline{s}(\sigma_2(\lambda^{-1}\gamma^{-1},\omega)^{-1})i(u(\gamma\lambda)(\omega))=\\
=&\widetilde{\varphi}(\omega)\overline{s}(\sigma_2(\gamma^{-1},\omega)^{-1})\overline{s}(\sigma_2(\lambda^{-1},\gamma^{-1}\omega)^{-1})i(-\overC^2(\sigma_2)(\overline{\varepsilon})(\gamma,\lambda)(\omega)+u(\gamma\lambda)(\omega))=\\
=&\overline{s}(\sigma_1(\gamma^{-1},\omega)^{-1})\widetilde{\varphi}(\gamma^{-1}\omega)\overline{s}(\sigma_2(\lambda^{-1},\gamma^{-1}\omega)^{-1})i(-\overC^2(\sigma_2)(\overline{\varepsilon})(\gamma,\lambda)(\omega)+u(\gamma\lambda)(\omega)-u(\gamma)(\omega)) \ ,
\end{align*}
where we used the same strategy as in the previous computation. By comparing the two computations we get
$$
\overC^2(\sigma_1)(\overline{\varepsilon})-\overC^2(\sigma_2)(\overline{\varepsilon})=\overline{\delta}^1u \ .
$$
Thus $\textup{H}^2(\sigma_1)(e_{\bbZ})=\textup{H}^2(\sigma_2)(e_{\bbZ})$. To prove that the same thing holds for their bounded versions we need to show that $u$ is bounded.  This is a consequence of the fact that 
$$
\overline{s}(\sigma_1(\gamma^{-1},\omega)^{-1})\widetilde{\varphi}(\gamma^{-1} \omega)(0) \ , \ \widetilde{\varphi}(\omega)\overline{s}(\sigma_2(\gamma^{-1},\omega)^{-1})(0) \in [0,2) \ ,
$$
since $\widetilde{\varphi}(\omega)(0) \in [0,1)$ for almost every $\omega \in \Omega$ thanks to our choice of lift. Hence $u$ is bounded and it holds 
$$
\overC_b^2(\sigma_1)(\overline{\varepsilon})-\overC_b^2(\sigma_2)(\overline{\varepsilon})=\overline{\delta}^1u \ .
$$
which implies the main statement. This concludes the proof. 
\end{proof}

This shows ($3$) of Theorem \ref{teor:main:intro} and concludes its proof. We are finally ready to sum up our considerations about parametrized Euler class and semicohomology in the following 

\begin{teor}\label{teor:ghys:nonvanishing}
Let $\Gamma$ be a countable group and let $(\Omega,\mu)$ be an ergodic standard Borel probability $\Gamma$-space. Given $\sigma_1,\sigma_2:\Gamma \times \Omega \rightarrow H$ two measurable cocycles, assume that $\sigma_1$ and $\sigma_2$ are semicohomologous. Then it holds
$$
\textup{H}^2_b(\sigma_1)(e^b_{\bbZ})=\textup{H}^2_b(\sigma_2)(e^b_{\bbZ}) \ .
$$
\end{teor}

\begin{proof}
If $\sigma_1$ admits an equivariant family of points then the statement boils down to Theorem \ref{teor:ghys:vanishing}. If $\sigma_1$ does not admit any equivariant family of points, the claim is obtained from Proposition \ref{prop:ghys:nonvanishing}
\end{proof}

\section{Minimal cocycles}\label{sec:minimal:cocycles}

In the previous section we saw that the parametrized Euler class is a complete invariant of the semicohomology class of a measurable cocycle. In general one could be more interested in understanding when two measurable cocycles are cohomologous rather than semicohomologous. In this section we are going to introduce the notion of \emph{minimality} given by Furstenberg \cite{furst:articolo} and we are going to prove that two ergodic minimal cocycles have the same parametrized Euler class if and only if they are cohomologous. 

\begin{deft}\label{def:minimal:cocycle}
Let $\Gamma$ be a countable group and let $(\Omega,\mu)$ be a standard Borel probability $\Gamma$-space. Consider a measurable cocycle $\sigma:\Gamma \times \Omega \rightarrow H$. Let $\mathcal{F}_0(\bbS^1)$ be the collection of not-empty closed subsets of $\bbS^1$ endowed with the Borel structure coming from the Hausdorff topology. We say that $\bbS^1$ is $\sigma$-\emph{minimal} or equivalently that $\sigma$ is \emph{minimal} if the only measurable function 
$$
\Phi:\Omega \rightarrow \mathcal{F}_0(\bbS^1) \ , 
$$
which is $\sigma$-equivariant, that is $\Phi(\gamma \omega)=\sigma(\gamma,\omega)\Phi(\omega)$, is the constant function $\Phi(\omega)=\bbS^1$.
\end{deft}

For ergodic minimal cocycles we are going to show that semicohomology and cohomology coincide.

\begin{prop}\label{prop:minimal:semicohomology}
Let $\Gamma$ be a countable group an let $(\Omega,\mu)$ be an ergodic standard Borel probability $\Gamma$-space. Given two measurable cocycles $\sigma_1,\sigma_2:\Gamma \times \Omega \rightarrow H$, suppose that they are minimal. Then $\sigma_1$ is left semicohomologous to $\sigma_2$ if and only if it is cohomologous to $\sigma_2$. 
\end{prop}

\begin{proof}
If two minimal cocycles are cohomologous, then they are also semicohomologous. So in particular $\sigma_1$ is left semicohomologous to $\sigma_2$. 
Clearly the relevant implication is the other one. 

Suppose that $\sigma_1$ is left semicohomologous to $\sigma_2$. By the minimality assumption $\sigma_1$ cannot admit an equivariant family of points. We know that there exists a measurable function
$$
\varphi:\bbS^1 \times \Omega \rightarrow \bbS^1 \ , 
$$
such that the slice $\varphi(\omega)$ is a non-decreasing degree one map and 
\begin{equation}\label{eq:left:semicohomology:local}
\varphi(\gamma \omega)\sigma_2(\gamma,\omega)=\sigma_1(\gamma,\omega)\varphi(\omega) \ ,
\end{equation}
for every $\gamma \in \Gamma$ and almost every $\omega \in \Omega$. If we define 
$$
E(\omega):=\overline{\textup{Im}(\varphi(\omega))} \ ,
$$
as the closure of the image of $\varphi(\omega)$, we have that 
$$
\Phi:\Omega \rightarrow \mathcal{F}_0(\bbS^1) \ , \ \ \ \Phi(\omega):=E(\omega) \ ,
$$
is a $\sigma_1$-equivariant closed-valued measurable function (the $\sigma_1$-equivariance follows by Equation \eqref{eq:left:semicohomology:local}). By the minimality of $\sigma_1$ we must have $E(\omega)=\bbS^1$ for almost every $\omega \in \Omega$. This means that the image $\textup{Im}(\varphi(\omega))$ is dense for almost every $\omega \in \Omega$. 

Denote by $\widetilde{\varphi}:\bbR \times \Omega \rightarrow \bbR$ a lift of $\varphi$. Since $\varphi(\omega)$ has dense image in $\bbS^1$, also $\widetilde{\varphi}(\omega):\bbR \rightarrow \bbR$ will have dense image. Additionally, since $\widetilde{\varphi}(\omega)$ is non-decreasing and commutes with integer translations, it must be continuous and surjective (for almost every $\omega$). Thus the same is true for the map $\varphi(\omega)$.

We need only to show that $\varphi(\omega)$ is injective for almost every $\omega \in \Omega$. In that case $\varphi(\omega) \in H$ and the statement will follow. By contradiction suppose that there exists a positive measure subset $A \subset \Omega$ such that $\varphi(\omega)$ is not injective for every $\omega \in \Omega$.
Define now 
$$
A_0:=\{ \omega \in \Omega \ | \ \textup{$\varphi(\omega)$ is not injective} \} \ .
$$
The above set is $\Gamma$-invariant. Indeed suppose that $\omega \in A_0$. This means that there exist distinct points $x,y \in \bbS^1$ such that 
$$
\varphi(\omega)(x)=\varphi(\omega)(y) \ . 
$$
We claim that $\gamma \omega \in A_0$. We need to show that $\varphi(\gamma\omega)$ is not injective. Notice that 
\begin{align*}
\varphi(\gamma \omega)(\sigma_2(\gamma,\omega)(x))&=\sigma_1(\gamma,\omega)\varphi(\omega)(x)=\\
&=\sigma_1(\gamma,\omega)\varphi(\omega)(y)=\varphi(\gamma \omega)(\sigma_2(\gamma,\omega)(y)) \ ,
\end{align*}
where we exploited both Equation \eqref{eq:left:semicohomology:local} and the fact that $\omega \in A_0$. Thus $A_0$ is a $\Gamma$-invariant positive measure subset of $\Omega$, and by ergodicity it must have full measure. This means that $\varphi(\omega)$ is not injective for almost every $\omega \in \Omega$. 

Since $\varphi(\omega)$ is not injective for almost every $\omega$, there exist distinct points $x_\omega,y_\omega \in \bbS^1$ such that 
$$
\varphi(\omega)(x_\omega)=\varphi(\omega)(y_\omega) \ . 
$$
If we now lift such points, we get points $\widetilde{x}_\omega,\widetilde{y}_\omega \in \bbR$ and we can suppose $\widetilde{x}_\omega < \widetilde{y}_\omega < \widetilde{x}_\omega+1$. By the properties of the lift $\widetilde{\varphi}(\omega)$ is must hold either $\widetilde{\varphi}(\omega)(\widetilde{x}_\omega)=\widetilde{\varphi}(\omega)(\widetilde{y}_\omega)$ or $\widetilde{\varphi}(\omega)(\widetilde{x}_\omega)=\widetilde{\varphi}(\omega)(\widetilde{y}_\omega)+1$. In both cases, there exists an open interval $I(\omega) \subset (\widetilde{x}_\omega,\widetilde{x}_\omega+1)$ on which $\widetilde{\varphi}(\omega)$ is constant. By projecting such a set to the circle, there exists a non-empty open interval on which $\varphi(\omega)$ is constant, for every $\omega \in \Omega$. We define
$$
U(\omega):=\{ x \in \bbS^1 \ | \ \textup{$x$ admits an open neighborhood $U$ such that $\varphi(\omega)|_U$ is constant} \} \ .
$$
By definition this is an open set. Additionally it is not empty by what we have proved so far and it cannot coincide with the whole $\bbS^1$ otherwise $\varphi(\omega)$ would be constant (a contradiction to the surjectivity). 
Using Equation \eqref{eq:left:semicohomology:local}, it follows that
\begin{equation}\label{eq:sigma:equivariance}
U(\gamma\omega)=\sigma_2(\gamma,\omega)U(\omega) \ .
\end{equation}
Hence we can define a measurable function
$$
\Psi:\Omega \subset \Omega \rightarrow \mathcal{F}_0(\bbS^1) \ , \ \ \ \Psi(\omega):=\bbS^1 \setminus U(\omega) \ .
$$
By Equation \eqref{eq:sigma:equivariance} the above function $\Psi$ is $\sigma_2$-equivariant and by the minimality assumption on $\sigma_2$ we must have $\Psi(\omega)=\bbS^1$ for almost every $\omega \in \Omega$. But this is a contradiction to the fact that $U(\omega)$ is not empty for almost every $\omega \in \Omega$. Thus $\varphi(\omega)$ is injective and this concludes the proof. 
\end{proof}

\begin{oss}
Since we are talking about minimal cocycles and Euler classes, we want to point out a mistake made by the author together with Moraschini in \cite[Theorem 5]{moraschini:savini}. In that paper we introduce the notion of Euler invariant associated to a measurable cocycle of a hyperbolic lattice $\Gamma \leq \textup{PSL}(2,\bbR)$. Such an invariant is obtained by integrating along $\Omega$ the real parametrized Euler class and pairing it with the fundamental class of the surface $\Gamma \backslash \bbH^2_{\bbR}$. 

In \cite[Theorem 5]{moraschini:savini} we claimed that the Euler invariant has bounded absolute value and it is maximal if and only if it is cohomologous to a hyperbolization. This statement is not true, and it should be modified by saying that maximal measurable cocycles are actually semicohomologous to a hyperbolization. Indeed in the proof we wrote that the twisting function which realizes the cohomology has image contained in $H$, but this is true only if the cocycle is minimal (as proved in Proposition \ref{prop:minimal:semicohomology}). 

This for instance happens when the measurable cocycle is associated to an integrable selfcoupling, as proved by Bader, Furman and Sauer \cite[Lemma 2.5]{sauer:articolo}. 
\end{oss}

\section{Vanishing of the real parametrized Euler class}\label{sec:real:euler:class}

In this section we are going to study the case when the parametrized Euler class vanishes as an essentially bounded real class. We are going to show that the class $\textup{H}^2_b(\sigma)(e^b_{\bbR})$ vanishes if and only if the cocycle is semicohomologous to a measurable cocycle taking values in the rotations subgroup $\textup{Rot} \subset H$. 

Before studying the vanishing of parametrized classes, we want to introduce the notion of elementary measurable cocycle and see how it is related the fact that a cocycle takes values in $\textup{Rot}$.

\begin{deft}\label{def:elementary:cocycle}
Let $\Gamma$ be a countable group and let $(\Omega,\mu)$ be a standard Borel probability $\Gamma$-space. A measurable cocycle $\sigma:\Gamma \times \Omega \rightarrow H$ is called \emph{elementary} if there exists a probability-valued map 
$$
\nu:\Omega \rightarrow \mathcal{M}^1(\bbS^1) \ ,
$$
which is weak-$^\ast$ measurable and $\sigma$-equivariant, that is 
$$
\nu(\gamma \omega)=\sigma(\gamma,\omega)_\ast\nu(\omega) \ .
$$
Here $\sigma(\gamma,\omega)_\ast$ denotes the push-forward action.
\end{deft}

We are going to see which consequences elementarity has on the dynamics determined by a measurable cocycle.

\begin{prop}\label{prop:consequence:elementary}
Let $\Gamma$ be a countable group and let $(\Omega,\mu)$ be an ergodic standard Borel probability $\Gamma$-space. Suppose that $\sigma$ is elementary. Then either there exists a measurable $\sigma$-equivariant map 
$$
F: \Omega \rightarrow \bbS^1_k \ ,
$$
where $\bbS^1_k$ denotes the collection of subsets of $\bbS^1$ with $k$ points, or $\sigma$ is semicohomologous to a cocycle with values in the rotations subgroup $\textup{Rot}$. 
\end{prop}

\begin{proof}
Since $\sigma$ is elementary, there exists a map $\nu:\Omega \rightarrow \mathcal{M}^1(\mathbb{S}^1)$ which is weak-$^\ast$ measurable and $\sigma$-equivariant. For almost every $\omega \in \Omega$, we can decompose $\nu(\omega)$ as
$$
\nu(\omega)=\nu_a(\omega)+\nu_{na}(\omega) \ ,
$$
where $\nu_a$ denotes the purely atomic part and $\nu_{na}$ denotes the non-atomic part. Additionally, the $\sigma$-equivariance of $\nu$ implies that 
$$
\nu_a(\gamma \omega)=\sigma(\gamma,\omega)_\ast\nu_a(\omega) \ ,
$$
where $\sigma(\gamma,\omega)_\ast$ denotes the push-forward measure. This tells us that the set
$$
A_0:=\{ \omega \in \Omega \ | \ \nu_a(\omega) \neq 0 \} \ 
$$
is a $\Gamma$-invariant measurable set and hence it has either full or null measure, by the ergodicity of $\Omega$. Suppose that it has full measure. We can define
$$
F:\Omega \rightarrow \bbS^1_k \ , \ \ F(\omega):=\{ x \in \bbS^1 \ | \ \nu_a(\omega)(x)=\max_{y \in \bbS^1} \nu_a(\omega)(y) \}  \ .
$$
Clearly $F(\omega)$ is a set of $k$ distinct points for some integer $k \geq 1$ and it is $\sigma$-equivariant by the equivariance of $\nu_a$. Notice that the fact that $k$ does not depend on $\omega \in \Omega$ relies once again on the ergodicity of the $\Gamma$-action. 

Suppose now that $A_0$ has null measure. This means that $\nu_a(\omega)=0$ and $\nu(\omega)=\nu_{na}(\omega)$ for almost every $\omega$. Identifying $\bbS^1 = \bbR / \bbZ$ we can define
$$
f:\bbR/\bbZ \times \Omega \rightarrow \bbR / \bbZ \ , \ \ f(x,\omega):=\nu(\omega)[0,x) \ \mod \bbZ \ .
$$
The map $f(\omega)(x):=f(x,\omega)$ is a non-decreasing degree one map which realizes the desired semiconjugacy. Indeed 
\begin{align*}
f(\gamma \omega)(\sigma(\gamma,\omega)(x))&=\nu(\gamma \omega)[0,\sigma(\gamma,\omega)(x)) \mod \bbZ=\\
&=\sigma(\gamma,\omega)_\ast\nu(\omega)[0,\sigma(\gamma,\omega)(x)) \mod \bbZ=\\
&=\nu(\omega)[\sigma(\gamma,\omega)^{-1}0,x) \mod \bbZ=\\
&=f(\omega)(x) -\nu(\omega)[0,\sigma(\gamma,\omega)^{-1}0) \mod \bbZ \ ,
\end{align*}
where we exploited that $\nu$ is $\sigma$-equivariant to move from the first line to the second one and then we concluded using the definition of push-forward measure. Thus $\sigma$ is semicohomologous to the cocycle
$$
\sigma_0:\Gamma \times \Omega \rightarrow \textup{Rot} \ , \ \ \sigma_0(\gamma,\omega)(x \mod \bbZ):=x-\nu(\omega)[0,\sigma(\gamma,\omega)^{-1}0) \mod \bbZ
$$
for $x \in \bbR$ and the claim is proved.
\end{proof}

The previous proposition has a cohomological counterpart. In fact the vanishing of the parametrized Euler class can detect those elementary cocycles taking values exactly into the rotations subgroup. 

\begin{proof}[Proof of Proposition \ref{prop:vanishing:real}]
We start supposing that $\textup{H}^2_b(\sigma)(e^b_{\bbR}) \in \textup{H}^2_b(\Gamma;\textup{L}^\infty(\Omega,\bbR))$ vanishes. Then there exists an essentially bounded function $u:\Gamma \times \Omega \rightarrow \bbR$ such that 
\begin{equation}\label{eq:real:coboundary}
\overC^2_b(\sigma)(\overline{\varepsilon})(\gamma,\lambda)(\omega)=\overline{\delta}^1u(\gamma,\lambda)(\omega) \ .
\end{equation}
Composing $u$ with the projection $\pi:\bbR \rightarrow \bbR / \bbZ$ we can consider the function $f:\Gamma \times \Omega \rightarrow \bbR / \bbZ$. As a consequence of Equation \eqref{eq:real:coboundary} we get that 
\begin{equation}\label{eq:reduced:measurable:cocycle}
\overline{\delta}^1(f)(\gamma,\omega)=0 \ \mod \ \bbZ \ ,
\end{equation}
that is $f$ is a cocycle. Thus, if we consider the measurable map
$$
\sigma_0:\Gamma \times \Omega \rightarrow \textup{Rot} \subset H \ , \ \ \ \sigma_0(\gamma,\omega)(x):=x+f(\gamma^{-1},\omega) \ ,
$$
Equation \eqref{eq:reduced:measurable:cocycle} implies that $\sigma_0$ is a measurable cocycle in the sense of Equation \eqref{eq:measurable:cocycle}. 

We want to compute the pullback of the Euler cochain along such a measurable cocycle. To do this, let $\overline{s}:H \rightarrow \widetilde{H}$ be the section defined by $\overline{s}(f)(0) \in [0,1)$. By what we have said so far it holds
$$
\overline{s}(\sigma_0(\gamma^{-1},\omega)^{-1})(x)=x-\{ u(\gamma,\omega) \} \ ,
$$
where $\{u(\gamma,\omega)\}$ is the fractionary part of $u(\gamma,\omega)$. Thus if we define 
$$
\{ u \}:\Gamma \times \Omega \rightarrow \bbR \ , \ \ \{ u \}(\gamma,\omega):=\{ u(\gamma,\omega) \} \ ,
$$
we get that
$$
\overC^2_b(\sigma_0)(\overline{\varepsilon})(\gamma,\lambda)(\omega)=\overline{\delta}^1(\{u\})(\gamma,\lambda)(\omega)\ .
$$
Hence we can rewrite Equation \eqref{eq:real:coboundary} as
$$
\overC^2_b(\sigma)(\overline{\varepsilon})(\gamma,\lambda)(\omega)=\overC^2_b(\sigma_0)(\overline{\varepsilon})(\gamma,\lambda)(\omega)+\overline{\delta}^1(\lfloor u \rfloor)(\gamma,\lambda)(\omega) \ ,
$$
where $\lfloor u \rfloor(\gamma, \omega):=\lfloor u(\gamma,\omega) \rfloor$ is the integer part of $u(\gamma,\omega)$. This means that 
$$
\textup{H}^2_b(\sigma)(e^b_{\bbZ})=\textup{H}^2_b(\sigma_0)(e^b_{\bbZ}) \ ,
$$
and by Theorem \ref{teor:ghys:pullback:semicohomology} it follows that $\sigma_0$ and $\sigma$ are semicohomologous. 

Viceversa, suppose that there exists a measurable cocycle $\sigma_0:\Gamma \times \Omega \rightarrow \textup{Rot} \subset H$ taking values in the rotations subgroup and such that $\sigma_0$ and $\sigma$ are semicohomologous. By Theorem \ref{teor:main:intro} we know that $\textup{H}^2_b(\sigma)(e^b_{\bbZ})=\textup{H}^2_b(\sigma_0)(e^b_{\bbZ})$, hence it follows $\textup{H}^2_b(\sigma)(e^b_{\bbR})=\textup{H}^2_b(\sigma_0)(e^b_{\bbR})$. It is sufficient to show that $\textup{H}^2_b(\sigma_0)(e^b_{\bbR})=0$ to get the statement. 

Since $\sigma_0$ takes values in the rotations subgroup and the latter is amenable, $e^b_{\bbR}|_{\textup{Rot}}=0$, whence the statement. 
\end{proof}

We conclude the section by studying some applications of the previous propositions. We will focus our attention on lattices in products and higher rank Lie groups. We will obtain in a different way results that are similar to both \cite[Theorem 1.10]{WZ01} for higher rank lattices and to \cite[Theorem 1.6]{BFS06} in the case of products.

Let $G_i$ be a locally compact group for $i=1,\cdots,k$ and let $\Gamma \leq \prod_{i=1}^k G_i$ be a lattice. We recall that $\Gamma$ is called \emph{irreducible} if the image $p_i(\Gamma)$ in $G$ is dense for every $i=1,\cdots,k$. Here $p_i:G \rightarrow G_i$ is the $i$-th projection. Similarly, given a standard probability $G$-space $\Omega$, we say that $\Omega$ is \emph{irreducible} if $G'_i:=\prod_{j \neq i} G_j$ acts ergodically on $\Omega$ for every $i=1,\cdots,k$. 

\begin{proof}[Proof of Theorem \ref{teor:lattice:product:elementary}]
By Proposition \ref{prop:vanishing:coefficients} it is sufficient to show that 
$$
\textup{H}^2_b(\Gamma;\textup{L}^\infty(\Omega,\bbR))=0 \ .
$$
By \cite[Corollary 9]{burger2:articolo} we have an injection 
$$
\textup{H}^2_b(\Gamma;\textup{L}^\infty(\Omega,\bbR)) \rightarrow \textup{H}^2_b(\Gamma;\textup{L}^2(\Omega,\bbR)) \ ,
$$
induced by the inclusion of the coefficients. Hence it is sufficient to show that $\textup{H}^2_b(\Gamma;\textup{L}^2(\Omega,\bbR))$ vanishes. 

Denote $G'_i:=\prod_{j \neq i} G_j$. As a consequence of \cite[Theorem 16]{burger2:articolo} it holds 
$$
\textup{H}^2_{cb}(\Gamma;\textup{L}^2(\Omega,\bbR)) \cong \bigoplus_{i=1}^k \textup{H}^2_{cb}(G_i; \textup{L}^2(\Omega;\bbR)^{G'_i}) \cong \bigoplus_{i=1}^k \textup{H}^2_{cb}(G_i;\bbR) \ ,
$$
where the second isomorphism holds because of the irreducibility of the space $\Omega$. Since we assumed that there is no Hermitian factor, we have that 
$\textup{H}^2_{cb}(G_i;\bbR)=0$ and this concludes the proof. 
\end{proof}

A similar behaviour holds in the case of higher rank lattices.

\begin{proof}[Proof of Theorem \ref{teor:higher:rank:elementary}]
As before, we want to exploit Proposition \ref{prop:vanishing:coefficients} to show that
$$
\textup{H}^2_b(\Gamma;\textup{L}^\infty(\Omega,\bbR))=0 \ .
$$
Since $\mathrm{L}^\infty(\Omega,\mathbb{R})$ is semi-separable, thanks to \cite[Corollary 1.6]{Mon10} we know that 
$$
\textup{H}^2_b(\Gamma;\textup{L}^\infty(\Omega,\bbR)) \cong \textup{H}^2_b(\Gamma;\textup{L}^\infty(\Omega,\bbR)^\Gamma) \cong \textup{H}^2_b(\Gamma;\bbR) \ ,
$$
by the $\Gamma$-ergodicity of $\Omega$. By \cite[Theorem 21]{burger2:articolo} the comparison map 
$$
\textup{comp}^2_\Gamma:\textup{H}^2_b(\Gamma;\bbR) \rightarrow \textup{H}^2_b(\Gamma;\bbR) \ ,
$$
is injective. Since we assumed $\textup{H}^2(\Gamma;\bbR)=0$, this concludes the proof. 
\end{proof}

\bibliographystyle{amsalpha}

\bibliography{biblionote}

\end{document}